\documentclass[a4paper,11pt]{article}      

\usepackage{graphicx}

\usepackage[english]{babel} 
\usepackage{amsmath} 
\usepackage{amsfonts}
\usepackage{amssymb}
\usepackage{amsthm}
\usepackage{amsxtra}
\usepackage{xspace}
\usepackage{mathrsfs}
\usepackage{xy}
\usepackage{xypic}

\newtheorem{Theorem}{Theorem}[section]
\newtheorem{Lemma}[Theorem]{Lemma}
\newtheorem{Corollary}[Theorem]{Corollary}
\newtheorem{Proposition}[Theorem]{Proposition}
\newtheorem{Definition}[Theorem]{Definition}

\newcommand{\Cset}{\mathbb{C}}

\newcommand{\Nset}{\mathbb{N}}

\newcommand{\Zset}{\mathbb{Z}} 

\newcommand{\qA}{\ensuremath{{\mathcal A}}\xspace}
\newcommand{\qB}{\ensuremath{{\mathcal B}}\xspace}
\newcommand{\qC}{\ensuremath{{\mathcal C}}\xspace}
\newcommand{\qD}{\ensuremath{{\mathcal D}}\xspace}
\newcommand{\qE}{\ensuremath{{\mathcal E}}\xspace}
\newcommand{\qF}{\ensuremath{{\mathcal F}}\xspace}
\newcommand{\qG}{\ensuremath{{\mathcal G}}\xspace}
\newcommand{\qH}{\ensuremath{{\mathcal H}}\xspace}
\newcommand{\qK}{\ensuremath{{\mathcal K}}\xspace}
\newcommand{\qL}{\ensuremath{{\mathcal L}}\xspace}
\newcommand{\qO}{\ensuremath{{\mathcal O}}\xspace}
\newcommand{\qP}{\ensuremath{{\mathcal P}}\xspace}
\newcommand{\qR}{\ensuremath{{\mathcal R}}\xspace}
\newcommand{\qS}{\ensuremath{{\mathcal S}}\xspace}
\newcommand{\qT}{\ensuremath{{\mathcal T}}\xspace}
\newcommand{\qU}{\ensuremath{{\mathcal U}}\xspace}

\newcommand{\qX}{\ensuremath{{\mathcal X}}\xspace}
\newcommand{\qY}{\ensuremath{{\mathcal Y}}\xspace}

\newcommand{\fg}{\ensuremath{\mathfrak{g}}\xspace} %
\newcommand{\ff}{\ensuremath{\mathfrak{f}}\xspace} %
\newcommand{\fh}{\ensuremath{\mathfrak{h}}\xspace} %
\newcommand{\fr}{\ensuremath{\mathfrak{r}}\xspace} %
\newcommand{\fs}{\ensuremath{\mathfrak{s}}\xspace} %
\newcommand{\fk}{\ensuremath{\mathfrak{k}}\xspace} %

\newcommand{\eins}{\ensuremath{{\rm 1\kern-.25em l}}\xspace}  % identity
 
\newcommand{\id}{\operatorname{id}}                        % id 

\begin{document}

\title{Weak Markov Processes as Linear Systems}

\author{Rolf Gohm \\        
Department of Mathematics \\
IMPACS, Aberystwyth University \\
email: rog@aber.ac.uk}           

\maketitle

\begin{abstract}

A noncommutative Fornasini-Marchesini system (a multi-variable version of a linear system) can be realized within a weak Markov process (a model for quantum evolution). For a discrete time parameter the resulting structure is worked out systematically and some quantum mechanical interpretations are given. We introduce subprocesses and quotient processes and then the notion of a $\gamma$-extension for processes which leads to a complete classification of all the ways in which processes can be built from subprocesses and quotient processes. We show that within a
$\gamma$-extension we have a cascade of noncommutative Fornasini-Marchesini systems. We study observability in this setting and as an application we gain new insights into stationary Markov chains where observability for the system is closely related to asymptotic completeness in a scattering theory for the chain. 
\end{abstract}



\section{Introduction}
\label{section:intro}

Recently there has been much attention for a certain multi-variable version of linear system theory which presents evolution equations of the form
\begin{eqnarray*}
x(\alpha k) &=& A_k\, x(\alpha) + B_k\, u(\alpha), \\
y(\alpha) &=& C\, x(\alpha) + D\, u(\alpha)\, .
\end{eqnarray*}
This is called a noncommutative Fornasini-Marchesini system in \cite{BGM05,BGM06,BV05}.
Here $\alpha \in F^+_d$ where $F^+_d$ is the free semigroup with $d$ generators (which we denote $1,\ldots,d$ and $d\in \Nset$ or $d=\infty$). The elements in $F^+_d$ are words in the letters $1,\ldots,d$, including the empty word $0$. Composition is defined by concatenation of words, for example $\alpha k$ in the formula above is the concatenation of a word
$\alpha$ and a generator $k \in \{1,\ldots,d\}$. Now $u, x, y$ are functions on $F^+_d$ taking values in vector spaces $\qU, \qX, \qY$ and $A_k: \qX \rightarrow \qX,\, B_k: \qU \rightarrow \qX,\, C: \qX \rightarrow \qY,\, D: \qU \rightarrow \qY$ are linear operators. For $d=1$
we have $F^+_d \simeq \Nset_0$ and this is the classical setting of linear system theory (state space models). The functions $u, x, y$ are interpreted as input, internal state, output of the system.

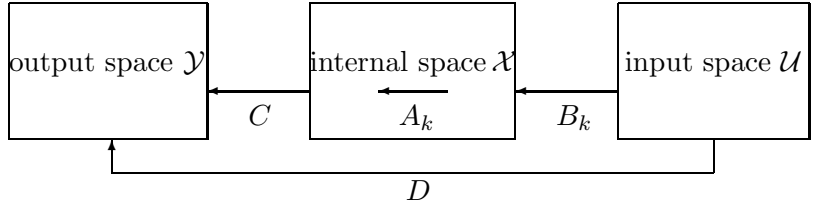
\setlength{\unitlength}{0.9cm}
\begin{picture}(15,3.5)

\put(1.0,1){\line(0,1){2}}
\put(1.0,1){\line(1,0){2.9}}
\put(3.9,1){\line(0,1){2}}
\put(1.0,3){\line(1,0){2.9}}

\put(5.4,1){\line(0,1){2}}
\put(5.4,1){\line(1,0){3}}
\put(8.4,1){\line(0,1){2}}
\put(5.4,3){\line(1,0){3}}

\put(9.9,1){\line(0,1){2}}
\put(9.9,1){\line(1,0){2.8}}
\put(12.7,1){\line(0,1){2}}
\put(9.9,3){\line(1,0){2.8}}

\put(5.4,1.7){\vector(-1,0){1.5}}
\put(7.4,1.7){\vector(-1,0){1.0}}
\put(9.9,1.7){\vector(-1,0){1.5}}

\put(1.0,2.0){output space $\qY$}
\put(10.0,2.0){input space $\qU$}
\put(5.4,2.0){internal space $\!\qX$}
%\put(1.5,2.1){$\cY$}
%\put(4.2,2.1){$\cH$}
%\put(8.5,2.1){$\cU$}
\put(4.5,1.2){$C$}
\put(6.7,1.2){$A_k$}
\put(9.0,1.2){$B_k$}

\put(2.5,0.5){\line(1,0){8.8}}
\put(2.5,0.5){\vector(0,1){0.5}}
\put(11.3,0.5){\line(0,1){0.5}}
\put(6.8,0.1){$D$}

\end{picture}

It has been established that many important mathematical concepts and results in linear system theory generalize nicely for all $d$, see \cite{BBF11}. We mention two such concepts which will later be studied in this paper. The first is the {\it observability map}
\[
\qO_{C,A} := (C \, A^\alpha)_{\alpha \in F^+_d}
\]
which is a linear map from $\qX$ into the $\qY$-valued functions on $F^+_d$.
(In this paper we use the convention that for
any variables $X_1, \ldots ,X_d$ and a word $\alpha = \alpha_1 \ldots \alpha_n \in F^+_d$ we have
$X_\alpha := X_{\alpha_1} \ldots X_{\alpha_n}$,
\; $X^\alpha := X_{\alpha_n} \ldots X_{\alpha_1}$, if $\alpha = 0$ we interpret it as an identity.)
By studying the observability map we can find out what we are able to know about the internal space $\qX$
by observing the output.

The second concept we want to mention here is the {\it transfer function} which is a description of how inputs are transferred into outputs.  
In the multi-variable context above the transfer function can be defined as a formal
power series 
\[
\qT(z) :=
\qT(z_1,\ldots,z_d) 
:= \sum_{\alpha \in F^+_d} \qT^\alpha z^\alpha 
:= D + C \sum^\infty_{r=1} (zA)^{r-1} zB
= D + C (I_\qX - zA)^{-1} zB
\]
where $z = (z_1\,I_\qX, \ldots ,z_d\,I_\qX)$
with indeterminates
$z_1,\ldots,z_d$ freely noncommuting among each other but commuting with the linear maps, $A = (A_1, \ldots, A_d)^t,$ 
$\; B = (B_1, \ldots, B_d)^t$, and $t$ denoting a transpose of the row vectors so $A$ and $B$ are column vectors of linear maps. In the classical case $d=1$ we call the single variable $z$ and then this reduces to a familiar formula which gives the transfer function as an analytic function. Explicitly
\[
\qT^\alpha = 
\left\{
\begin{array}{cc}
D & \text{if}\; \alpha = 0 \\ 
C\, B_\alpha & \text{if}\; |\alpha| = 1 \\
C\, A_{\alpha_r} \ldots A_{\alpha_2}\,B_{\alpha_1} & \text{if}\;  \alpha = \alpha_1 \ldots \alpha_r,\, r = |\alpha| \ge 2
\end{array} 
\right.
\]
We come back to this at the end of Section 3.
%\ref{section:structure}
See also \cite{BBF11,Go12}
for further discussions of similar formulas.   

Mathematical system theory is an abstraction from the physical dynamics. But of course its relevance to the real world depends on the fact that such a physical dynamics exists in the background. The basic idea behind our work comes from the observation that a non-commutative multi-variable system theory such as the one sketched above arises quite naturally from processes describing the evolution of quantum systems. 
There was some motivation for a non-commutative multi-variable system theory from the theory of formal languages and multi-scale systems in \cite{BGM05,BGM06,MB10}. But the project to investigate the connections with quantum dynamical systems was started in \cite{Go09}, compare further \cite{DH14} for a recent generalization of the model in \cite{Go09} and also \cite{Go12} for related work. 
In these papers the quantum processes are based on infinite tensor product constructions which are natural from the point of view of approximating the Fock spaces in continuous time physical processes. 
Compare also \cite{GK15} for an input-output formalism of quantum Markov dynamics based on a tensor product model.

To see the connections to operator and system theory more directly and on a more fundamental level we propose here to start with the concept of `weak Markov processes', worked out by Bhat and Parthasarathy in \cite{BP94,BP95} to catch the most fundamental features of quantum Markov processes.
In fact the connections to operator theory are very direct here because the concept of a weak Markov process can be interpreted as operator theoretic dilation theory studied from a probabilistic point of view. The dilation theory will be mentioned only in side remarks in this paper however, the emphasis lies on a development of the structure theory of the processes and on the interpretation of this structure. 
The benefits of such studies go in both directions: access to operator and system theory tools for the investigation of concrete quantum models on the one hand, guidance for the development of general system theory from the questions arising in quantum models on the other hand. But to be able to do that we need to define the relevant concepts and to develop a more systematic mathematical theory. The following sketch of the contents of this paper should be read with this motivation in mind.

In Section 2
%\ref{section:weak} 
we start, for convenience, with a self-contained but rather concise description of the basic theory of weak Markov processes in discrete time. The dynamics is described by a $*$-endomorphism or, equivalently, by a row isometry. This produces a one-sided time evolution which exhibits features related to causality
and system theory more directly than other approaches. 
By additionally considering a co-invariant subspace more such features emerge which further can be given a quantum probabilistic interpretation, such as transition operators and weak filtrations. Most of this is well known but at this point there is a need to work out a kind of dictionary between quantum probabilists with their nonspatial view of processes in terms of operator algebras, quantum channels and completely positive maps on the one hand and operator theorists with their spatial view focussing on operators acting on Hilbert spaces
on the other hand. The actual physical content is a third aspect to be considered. Note in particular how in the end of Section 2 
%\ref{section:weak}
we give an operational meaning to the elements of the free semigroup $F^+_d$ by interpreting them as certain measurement protocols. 

In Section 3
%\ref{section:structure} 
we define the notion of a representation of structure maps
$A,B,C,D$ (as described above) by a weak Markov process and in this way we get an
explicit systematic procedure to identify multi-variable linear systems (as described above) within quantum physical models. 
At this point it remains quite abstract but we go on to develop some quantum mechanical interpretation in terms of conditional states and quantum filtering. As a preparation for seeing specific representations of structure maps in quantum physical processes we develop in Section 4 %\ref{section:cat}
a theory of subprocesses and quotient processes
of (discrete weak Markov) processes and then show that in a suitably defined category of processes this can be reformulated as a short exact sequence. 
The main result here is a classification of extensions appearing in such short exact sequences by a construction which we call a $\gamma$-extension of processes. There is a set of contractions from which $\gamma$ can be chosen which gives a parametrization of all the ways in which two processes can be put together as subprocess and quotient process, with $\gamma = 0$ yielding the direct sum. In fact 
Section 4 can also be read as a more or less self-contained theory on its own with a lot of potential for further development beyond the rather specific use we make of it in the following sections. 

In Section 5
%\ref{section:cascades} 
we show that within a $\gamma$-extension of processes we have a representation of a cascade of the original systems. Let us remark here that the notion of cascades and of more general quantum networks of systems and processes has been around for some time in a continuous time setting and this theory has been investigated intensely because of promising applications in quantum filtering and quantum control \cite{GGY08,GJ09a}. It is not identical with the input-output formalism in this paper and a detailed investigation of connections between the theory of $\gamma$-extensions and such networks has yet to be undertaken.
There is no explicit work on continuous time systems in this paper, however we provide the basis for such a comparison by describing weak Markov processes in terms of product systems which suggests how to build the theory starting from continuous product systems. Our justification for concentrating on the discrete time setting here is the same as in Helton's seminal paper 
\cite{He72}
in which he connects classical system theory ($d=1$) with operator models and scattering theory, saying: `We concentrate on discrete time systems because it is for these that the relationship is most clear' (\cite{He72}, p.15).

In fact, similar to the path followed by Helton in \cite{He72} it is quite natural in our setting to investigate connections between observability of represented multi-variable linear systems and scattering theory for quantum physical models. We study observability in Section 6 %\ref{section:observable}
and show that in a quantum model it can be given the same interpretation as in classical system theory, namely that by observing the outputs we have an indirect way to measure and to investigate those parts of the internal space which have the character of a black box. 
While the choice of an input space is always rather canonical in our setting we argue that we also have a rather canonical choice for an output space if the process is a $\gamma$-extension: here the input space of the subprocess provides an interesting output space for the process. Looking for observability in this situation amounts to the question how much we can find out about the process by observations which only involve the subprocess.
The extreme case when we can find out everything
about the process by such special observations we call observable by the subprocess.

While in Section 6 all this is examined rather from the spatial point of view taken by an operator theorist, in Section 7 we confront it with a nonspatial approach provided by a quantum Markov chain given in an operator algebraic setting. The connection comes from the observation that to a stationary quantum Markov chain (which includes also an invariant state) we can associate a dual weak Markov process together with a subprocess (with $1$-dimensional internal space)
and then we can apply the techniques established earlier. We discuss an example and illustrate a quantum physical interpretation of the recursions in the noncommutative Fornasini-Marchesini system in terms of quantum filtering and quantum tomography. Some work remains to be done here to make the connection with the original literature on these topics more explicit. 

Finally we show that observability of the dual weak process by this subprocess is equivalent to asymptotic completeness in a scattering theory for (operator algebraic) stationary quantum Markov chains first introduced by K\"{u}mmerer and Maassen in \cite{KM00}. For illustration we write down the M{\o}ller operator but the actual construction and the details of this scattering theory need a setting with two-sided time evolution and we refer to the literature for these details. We finish this paper with a somewhat sketchy overview and a discussion about work on the corresponding operators on the level of weak processes, again giving suitable references for the reader who wants to get deeper into this. 

One of the reasons why we consider this equivalence to be important on a conceptual level is that the
scattering theory in \cite{KM00} is motivated by Lax-Phillips scattering theory \cite{LP67} and proceeds to construct an operator-algebraic analogue but here it becomes clear that it is really more than just an analogue: we can actually go to a kind of multi-variable generalization of Lax-Phillips scattering theory. 
See in particular \cite{BV05} and also \cite{Go09} on this topic. It would be interesting to investigate if better computational procedures can be developed based on these insights. 

For example the criterion for observability by a subprocess in terms of the transition operator $Z$ established in Section 6 is a generalization of a criterion for asymptotic completeness of stationary Markov chains obtained in 
\cite{Go04a,GKL06}.
This is an excellent case of the cross-fertilization between quantum probability on the one hand and operator and system theory on the other hand which we have in mind.
It is based on a very special case of our general theory, subprocesses with $1$-dimensional internal space, and it is reasonable to assume that much more can be achieved here by future work.

One of the referees for this paper seemed somewhat disappointed that we didn't get a closer structural match with the results of Helton in the already mentioned \cite{He72}. My following comment on that may be of wider interest for readers of this paper. As explained above, what we achieve is a realization of linear system theory concepts from a quantum dynamics running in the background. Quantum mechanics imposes certain interpretations on us which should be consistent with the interpretation of a system theory concept such as observability. For example, if you compare the dynamics by which Helton in \cite{He72}, Section 2, backs up his structure maps $A,B,C,D$ with the corresponding Definition \ref{def:structure} in our paper then you notice that we aim for a direct approximation of the system state by output states within the quantum process and in real physical time and for this reason we do not choose the output spaces orthogonal to the system space, as Helton does. For different purposes one can think about other choices, and in fact the sketch of scattering theory results at the end of Section 7 makes it plausible that in a systematic analysis of two-sided processes it should be possible to obtain a closer match with Helton's results.

\section{Weak Markov Processes}
\label{section:weak}

Unitary dynamics on a Hilbert space is the most basic way of describing quantum mechanical evolution. If causality is taken into account 
and one restricts the attention to observables belonging only to the future (or only to the past)
then it becomes natural to study $*$-endomorphisms
of $\qB(\qH)$, the algebra of bounded linear operators on a Hilbert space $\qH$, as opposed to $*$-automorphisms which implemented by unitaries come up in the Heisenberg picture of quantum mechanics and are well understood. A convincing argument in this direction is presented in \cite{Ar03}, Section 1.2. From now on let $\theta: \qB(\qH) \rightarrow \qB(\qH)$ be a $*$-endomorphism and $\qH$ a separable Hilbert space.  
Then it follows from the representation theory of $\qB(\qH)$ that there exists a separable Hilbert space $\qP$ and an isometry $V: \qH \otimes \qP \rightarrow \qH$ such that for all $X \in \qB(\qH)$ 
\[
\theta(X) = V \; X \otimes \eins_\qP \; V^*\,.
\]
We assume $\theta \not=0$, then $\theta$ is automatically injective. See \cite{La93} for more details. 
If $(\epsilon_k)^d_{k=1}$ (with $d \in \Nset$ or $d=\infty$) is an orthonormal basis of $\qP$ then with 
\[
V_k := V |_{\qH \otimes \epsilon_k} \in \qB(\qH),
\quad k=1,\ldots,d,
\]
we can also write
\[
\theta(X) = \sum^d_{k=1} V_k X V^*_k
\]
(limits to be understood in the strong operator topology if $d=\infty$) which is called a {\it Kraus decomposition}.  The $V_k$ are isometries with orthogonal ranges, so alternatively (and with the same notation) we also think of $V$ as a {\it row isometry}
\[
V = (V_1, V_2, \ldots, V_d): \; \bigoplus^d_1 \qH \rightarrow \qH\,.
\]
$V$ is called a {\it row unitary} if $\sum^d_{k=1} V_k V^*_k = \eins$
and this is equivalent to $\theta$ being unital, i.e., $\theta(\eins)=\eins$. If this additional assumption is used in the following then it will always be explicitly stated.
\\
\\
Using induction we define row isometries $V^{(n)} \colon \qH \otimes \bigotimes^n_1 \qP
\rightarrow \qH$ 
(where $n\in \Nset$) by $V^{(1)} := V$ and, for $\tilde{\xi} \in \qH \otimes \bigotimes^{n-1}_1 \qP$ and $\eta \in \qP$
\[
V^{(n)} (\tilde{\xi} \otimes \eta)
= V \big( V^{(n-1)}\tilde{\xi} \otimes \eta \big)\,.
\]
Later we need the following equivalent description for the adjoints: If $\xi \in \qH$
and $V^* \xi =: \sum_k \xi_k \otimes \eta_k \in \qH \otimes \qP$ then
\[
V^{(n)*} \xi = \sum_k V^{(n-1)*} (\xi_k) \otimes \eta_k
\in (\qH \otimes \bigotimes^{n-1}_1 \qP) \otimes \qP = \qH \otimes \bigotimes^n_1 \qP\,.
\]
It is not difficult to check that this implements the $n$-th power of $\theta$:
\[
\theta^n(X) = V^{(n)}\; X \otimes \eins_{\bigotimes^n_1 \! \qP}\;
V^{(n)*} = \sum_{\alpha \in F^+_d, |\alpha|=n} V_\alpha X V^*_\alpha\,.
\]
Here the notation is $V^*_\alpha
:= (V_\alpha)^* = V^*_{\alpha_n} \ldots V^*_{\alpha_1}$ if $\alpha = \alpha_1 \ldots \alpha_n$ and $|\alpha|=n$ is the length of the word. 
This Kraus decomposition where a sum over all words of length $n$ occurs gives the first connection to the multi-variable formalism sketched in Section 1.
%\ref{section:intro}
\\
Remark: The tensor products
$\big( \bigotimes^n_1 \qP \big)_{n \in \Nset}$ appearing in these formulas represent the (discrete) product system associated to the endomorphism $\theta$ and though we do not go into continuous time systems in this paper it is worth noting that the natural starting point to translate our results to continuous time would be to consider continuous product systems and in this way make the connection with the theory exposed in \cite{Ar03,Bh01,Go06}. 

To introduce processes which resemble Markov processes from probability theory we need to specify a subspace $\fh \subset \qH$ (by which we always mean a closed subspace if not otherwise stated).
Let us denote the orthogonal projection onto $\fh$ by $p = P_{\fh}$.
(In this paper we use consistently the notation $P_\qL$ for the orthogonal projection onto a subspace $\qL$.)
Given $\fh \subset \qH$ we have (with $n \in \Nset_0$) a family of normal $*$-homomorphisms
\begin{eqnarray*}
J^{(n)}: \qB(\fh) &\rightarrow& \qB(\qH) \\
x &\mapsto& \theta^n (x p)
\end{eqnarray*}
and the compressions $Z_n: \qB(\fh) \rightarrow \qB(\fh)$ defined by
\[
p\, J^{(n)}(x)\, p =: Z_n(x)\, p \quad  \big[ = J^{(0)}(Z_n(x)) \big]
\]
which are contractive completely positive maps. For the processes to be defined below the $J^{(n)}$ play the role of {\it non-commutative random variables} and the $Z_n$ are {\it transition operators}. 

The subspace $\fh$ of $\qH$ is called {\it invariant} if $V_k \fh \subset \fh$ for all $k=1,\ldots d$ and {\it co-invariant} 
if $V^*_k \fh \subset \fh$ for all $k=1,\ldots d$. 
The importance of co-invariant subspaces in this context has been observed by many, see for example \cite{BJKW00} for various related topics. We note some useful properties equivalent to co-invariance. 

\begin{Lemma} \normalfont \label{lem:co}
The following are equivalent:
\begin{itemize}
\item[(1)]
$\theta(p)\, p = \theta (\eins)\, p$\,.
\item[(2)]
$\fh \perp V (\fh^\perp \otimes {\mathcal P})$.
\item[(3)]
$\fh$ is co-invariant.
\end{itemize}
\end{Lemma}

Let us write $Z$ for $Z_1$ and state the following modification
of the previous lemma:

\begin{Lemma} \normalfont \label{lem:co+}
The following are equivalent:
\begin{itemize}
\item[(1')]
$p \le \theta (p)$\,.
\item[(2')]
$\fh \subset V (\fh \otimes {\mathcal P})$.
\item[(3')]
$p \le \theta(\eins)$ and $\fh$ is co-invariant.
\item[(4')]
$Z(\eins_{\fh}) = \eins_{\fh}$.
\end{itemize}
\end{Lemma}

Below we prove Lemma \ref{lem:co}. It is then easy to get
Lemma \ref{lem:co+} by checking that ($\ell'$) is nothing but
($\ell$) together with $p \le \theta(\eins)$ (for $\ell=1,\ldots,3$). The equivalence of (1') and (4') is
immediate because, by definition, $Z(\eins_{\fh}) = p\, \theta(p)|_{\fh}$. 

If $V$ is a row unitary then all the properties in Lemmas \ref{lem:co} and \ref{lem:co+} are equivalent; for example $(3)$ and
$(3')$ are equivalent for a row unitary because in this case we have $\theta(\eins)=\eins$.
This means that for a row unitary we can always use the simpler statements (1')-(4') when dealing with co-invariant subspaces. 

\begin{proof}
Consider an orthogonal projection $q$ onto $q \qH$. Then
$\theta(q) = V (q \otimes \eins_{\mathcal P}) V^*$ is the
orthogonal projection onto $V(q\qH \otimes {\mathcal P})$. Hence
$\theta(\eins - p)$ projects onto  $V(\fh^\perp \otimes \qP)$. This gives $(1) \Leftrightarrow (2)$. 

Applying $V^*$ to (2) we obtain
$V^* \fh \perp \fh^\perp \otimes {\mathcal P}$, hence
$V^* \fh \subset \fh \otimes {\mathcal P}$ and 
$V^*_k \fh \subset \fh$ for all $k=1,\ldots d$ which is (3). Conversely, from (3) we get
\[
\theta(p) p = \sum^d_1 V_k p V^*_k p
= \sum^d_1 V_k V^*_k p = \theta(\eins) p
\]
which is (1). 
\end{proof}

The following definition is consistent with the terminology used by Bhat and Parthasarathy in \cite{BP94,BP95} where also continuous time and versions with $C^*$-subalgebras of $\qB(\fh)$ are considered which allows the inclusion of classical Markov processes by restricting to commutative subalgebras. We only consider discrete time steps and focus on the algebra $\qB(\fh)$ of all bounded operators. This allows a lean formulation from which further (probabilistic) features of the processes can be pulled out afterwards.

\begin{Definition} \normalfont \label{def:weak}
A {\it minimal discrete weak Markov process} is a triple $(\qH, V, \fh)$ where $V$ is a row isometry on a Hilbert space $\qH$ which contains the Hilbert space $\fh$ as a co-invariant subspace with respect to $V$, 
and such that we have minimality, i.e., 
\[
\qH = \overline{span} 
\{ V_\alpha \fh \colon \alpha \in F^+_d \}\,.
\]
If $V$ is a row unitary then we call the process {\it unital}.
\end{Definition}

For simplicity in this paper we refer to minimal discrete weak Markov processes as processes.
Note that if $V$ is originally defined on a larger Hilbert space 
which contains the Hilbert space $\fh$ as a co-invariant subspace with respect to $V$
then we can always restrict to a space $\qH$ satisfying the additional minimality assumption which ensures that $\qH$ is the smallest invariant subspace containing $\fh$. Because $\fh$ is co-invariant it follows 
that $\qH$ is even reducing in the original larger Hilbert space.

Remark:
From the point of view of dilation theory $V$ is the minimal isometric dilation of its compression to $\fh$ and as such it is determined by it up to unitary equivalence, see \cite{Po89a}. 

Let us now investigate the following increasing sequence of subspaces which is called the weak filtration associated to the process: $q_k := \theta^k(p),\;
p_k := sup(q_0,\ldots,q_k),\; \fh_k := p_k \qH$ (for
$k \in \Nset_0$). Then $p=p_0 \le p_1 \le p_2 \le \ldots$ and $\fh = \fh_0 \subset \fh_1 \subset \fh_2 \subset \ldots$.

Note that if the process is unital then $q_k = p_k$ for all $k$ and the arguments simplify. For unital processes 
we have $\fh_n = V^{(n)} (\fh \otimes \bigotimes^n_1 \!\qP)$ and the
inclusion $\fh_n \subset \fh_{n+1}$ is mapped by $V^{(n+1)*}$ to an inclusion $\fh \otimes \bigotimes^n_1 \!\qP \subset
\fh \otimes \bigotimes^{n+1}_1\! \qP$. Hence in this case we have an identification of $\qH$ with the inductive limit of the sequence $\big( \fh \otimes \bigotimes^n_1 \! \qP \big)$. 

\begin{Proposition} \normalfont \label{prop:weak}
Let $(\qH, V, \fh)$ be a process. Then 
\begin{itemize}
\item[(a)]
$V(\fh_m \otimes {\mathcal P}) \subset \fh_{m+1}$ (for all $m \in \Nset_0$)
\item[(b)]
$Z_n = Z^n$ (for all $n \in \Nset$)
\item[(c)]
If $0 \le m \le n$ then $p_m V^{(n)} = q_m V^{(n)}$.
\item[(d)]
If $0 \le m \le n$ then
\[
p_m J^{(n)}(x) p_m = J^{(m)}(Z^{n-m}(x))\quad\quad
\]
\end{itemize}
\end{Proposition}

(a) states an adaptedness property for the process and the filtration.
(b) means that $n \mapsto Z_n$ defines a semigroup. We can think of it as a nonspatial analogue of the Lax-Phillips contraction semigroup 
\cite{LP67} or as a noncommutative analogue of the Chapman-Kolmogorov semigroup of transition operators for classical Markov processes \cite{Fe68}. 
In fact (b),(c) and (d) resemble properties of Markov processes in classical probability and hence motivate the terminology `weak Markov process'. Versions of (d) appear in Bhat's papers as the `weak Markov property'. 

\begin{proof}
$\;\theta(p_m)$ is the projection onto
$V(\fh_m \otimes {\mathcal P})$. Hence (a) is nothing but a reformulation of the obvious $\theta(p_m) \le p_{m+1}$. 
We prove (b) by induction. For $n=1$ we have the definition of $Z$. Now suppose that for $n \ge 2$
\[
Z_{n-1}(x) = p\, V^{(n-1)}\, x \otimes \eins\; V^{(n-1)*}|_\fh = Z^{n-1}(x)\,.
\]
Then we can use (2) of Lemma \ref{lem:co} in the form
$p V = p V (p \otimes \eins)$ to get
\begin{eqnarray*}
Z_{n} (x) &=& p\, V^{(n)}\, x \otimes \eins\; V^{(n)*}|_\fh \\
&=& p\, V \big[ (V^{(n-1)} \otimes \eins)((x \otimes \eins) \otimes \eins) (V^{(n-1)*} \otimes \eins) \big] V^*|_\fh \\
&=& p\, V (p \otimes \eins) \big[ \ldots \big] (p \otimes \eins) V^*|_{\fh} \\
&=& p\, V (Z^{n-1}(x) \otimes \eins) V^*|_{\fh} = Z^{n}(x) \\
\end{eqnarray*}
To prove (c) we note that (2) of Lemma \ref{lem:co}
in the form $V(\fh^\perp \otimes {\mathcal P}) \perp \fh$
can be iterated to yield 
\[
V^{(\ell)} (\fh^\perp \otimes {\mathcal P}^{\otimes \ell}) \perp \fh\,.
\]
for all $\ell \in \Nset$. 
Hence for all $\ell \le m$ we find, by applying $V^{(m-\ell)}$, 
\[
V^{(m)} (\fh^\perp \otimes {\mathcal P}^{\otimes m})
\perp 
V^{(m - \ell)} (\fh \otimes {\mathcal P}^{\otimes (m - \ell)})
\] 
which implies
$
p_m V^{(m)} (p^\perp \otimes \eins) = 0$.

Together with $q_m = \theta^m(p) = V^{(m)}\, p \otimes \eins \,V^{(m)*}$ we obtain
\[
p_m V^{(m)} = p_m V^{(m)}\, p \otimes \eins = p_m q_m V^{(m)}
= q_m V^{(m)}\,.
\]
The iterative definition of the $V^{(m)}$ shows that their ranges do not increase if $m$ increases. Hence also $p_m V^{(n)} = q_m V^{(n)}$ whenever $m \le n$. 

To get (d) we start from the definition of $Z_{n-m}$,
\[
p\, J^{(n-m)}(x) p = Z_{n-m}(x)\, p\,, \quad\quad (m \le n),
\]
and apply $\theta^m$ to get
\[
q_m J^{(n)}(x) q_m = \theta^m (Z_{n-m}(x) p)
= J^{(m)}(Z_{n-m}(x))
\]
Now because of (b) and (c) we can replace $Z_{n-m}$ by
$Z^{n-m}$ and $q_m$ by $p_m$.
\end{proof}

The following observation, also noted in \cite{Bh96}, is crucial for our approach. Bhat remarks in \cite{Bh96}, p.562, in this context: `Roughly speaking there is also an additive structure when we deal with general quantum dynamical semigroups \ldots A detailed study of such systems is yet to be undertaken.' The work presented in this paper goes into this direction.

\begin{Proposition} \normalfont \label{prop:wandering}
Let $(\qH, V, \fh)$ be a process. Then
\[
\qE := \fh_1 \ominus \fh_0 = \overline{span} (\fh, V(\fh \otimes {\mathcal P}))
\ominus \fh
\]
is a wandering subspace, i.e., 
$V_\alpha \qE \perp V_\beta \qE$ if $\alpha \not= \beta$, and
\[
\qH = \fh \oplus \bigoplus_{\alpha \in F^+_d} V_\alpha \qE,
\quad
\fh_n = \fh \oplus \bigoplus_{|\alpha| < n} V_\alpha \qE 
\]
\end{Proposition}

\begin{proof}
To prove that $\qE$ is wandering it is enough to show that $\qE \perp V_\alpha \qE$ for all $\alpha \not= 0$. Suppose $\alpha \not= 0$. Because
$\qE \perp \fh$ we conclude by (2) of Lemma \ref{lem:co}
that $V_\beta \qE \perp \fh$ for all $\beta \in F^+_d$.
In particular $V_\alpha \qE \perp \fh$. Writing $\alpha = k \beta$ with $k \in \{1,\ldots,d\}$ and $\beta \in F^+_d$ we find $V_\alpha \qE \perp V_k \fh$. Finally if
$k' \in \{1,\ldots,d\}$ but $k'\not= k$ then
$V_\alpha \qE \perp V_{k'} \qH$ because $V_k$ and $V_{k'}$ have orthogonal ranges. Putting it all together we have $\qE \perp V_\alpha \qE$.
The other assertions are now immediate.
\end{proof}

The following results indicate how the $V_\alpha$ are related to the quantum physical behaviour of the process. For this we apply the standard interpretations of quantum physics to the mathematical objects. Suppose $X: \qH \rightarrow \qH$ and $Y: \qP \rightarrow \qP$ are linear operators. Then for $n \in \Nset_0$ and  $m \in \Nset$ we have the following linear operators on $\qH$

\begin{eqnarray*}
X_n &:=& V^{(n)} \; X \! \otimes \bigotimes^{n}_1 \!
\eins_\qP \; V^{(n)*}\,, \\
Y_m &:=& V^{(m)} \; \eins_\qH \! \otimes \bigotimes^{m-1}_1 \!
\eins_\qP \otimes Y \; V^{(m)*}
\end{eqnarray*}
(empty tensor products to be omitted for $n=0$ and $m=1$). We refer to these operators as observables.

\begin{Proposition} \normalfont \label{prop:operational}
Suppose that the process $(\qH, V, \fh)$ is unital. Then the observables $Y_m$ commute with each other and $Y_m$ commutes with $X_n$ whenever $m \le n$.

Now suppose further that $Y \epsilon_j = j \,\epsilon_j$ for the orthonormal basis $(\epsilon_j)^d_{j=1}$ of $\qP$ which we use to define the isometries $V_j$. 
If the process is prepared in a vector state given by a unit vector $\eta$ in the range of $V_\alpha$ so that the $m$-th letter in $\alpha$ is equal to $k$
then the measurement of $Y_m$ yields the outcome $k$ with certainty.
\end{Proposition}

\begin{proof}
For $1 \le m \le n$ we define
\[
Y_{m,n} := V^{(n)} \; \eins_\qH \otimes \bigotimes^{m-1}_1 \!
\eins_\qP \otimes Y \otimes \bigotimes^{n}_{m+1}\eins_\qP \; V^{(n)*}
\]
(empty tensor products omitted in the cases $m=1$ and $m=n$). 
Because the process is unital,
$V$ is a row unitary and in particular $V \, \eins_\qH \otimes \eins_\qP \, V^*=\eins_\qH$. It follows that
$Y_m = Y_{m,n}$ 
whenever $m \le n$. Writing $Y_m$ in this way it becomes obvious that it commutes with $X_n$ and $Y_n$.

Let $P_{\epsilon_k}$ be the orthogonal projection onto the one-dimensional subspace $\Cset \epsilon_k$ of $\qP$.
Then
\[
\langle \eta, V^{(m)} \; \eins_\qH \otimes \bigotimes^{m-1}_1 \!
\eins_\qP \otimes P_{\epsilon_k} \; V^{(m)*} \eta \rangle
=
\langle V^{(m)*} \eta, \eins_\qH \otimes \bigotimes^{m-1}_1 \!
\eins_\qP \otimes P_{\epsilon_k} \; V^{(m)*} \eta \rangle
\]
is the probability that a measurement of $Y_m$ yields the result $k$,
according to the standard rules of quantum mechanics. 
Now suppose that the process is prepared in a vector state with a unit vector $\eta$ in the range of $V_\alpha$, i.e., $\eta = V_\alpha \xi$ for some $\xi \in \qH$, so that the $m$-th letter in $\alpha$ is equal to $k$. Then from
\[
V^{(m)*} \eta = \sum_{|\beta|=m} V^*_\beta V_\alpha \xi \otimes \epsilon_\beta
\]
(with $\epsilon_\beta$ for $\beta = \beta_1 \ldots \beta_m$ is short for $\epsilon_{\beta_1}\otimes \ldots \epsilon_{\beta_m} \in  \bigotimes^{m}_1 \! \qP$)
we see that only the terms in the sum with $\beta_m = k$ can be non-zero and hence the application of $P_{\epsilon_k}$ at the $m$-th copy of $\qP$ always acts identical on all the non-zero terms. So the probability above is the squared length of a unit vector, i.e., it is equal to $1$.
\end{proof}

This result gives an operational meaning to the words $\alpha \in F^+_d$ by identifying them with measurement protocols for certain observables.
We can think of measuring $Y_1, \ldots, Y_n$ as performing a certain type $Y$ of measurement at the consecutive times $1, \ldots, n$ and the commutation properties proved above ensure that these measurements can be performed without perturbing the system (non-demolition measurements). For all $X \in \qB(\qH)$ it makes sense to 
consider $J^{(n)}(X) = X_n$ conditioned on $Y_1, \ldots, Y_n$ (conditioning in the sense of classical probability theory).
Such schemes and their application to quantum filtering and quantum control have been pioneered by Belavkin \cite{Be83,Be88} and this quantum filtering theory is a rapidly developing field of study. We mention the recent introductory surveys \cite{BHJ07,BHJ09} which give many references, the latter focusing on discretized models and containing constructions similar to the one above, see \cite{BHJ09}, Section 2.5. The quantum filtering equations for homodyne detections
or for photon counting described in Section 5 of \cite{BHJ09} refer to the measurement of observables of the type used in Proposition \ref{prop:operational}. Note that a different choice of the orthogonal basis in the Hilbert space $\qP$ corresponds to measurements of observables which do not commute with the original observables. So this choice of basis is part of the experimental set-up and different choices can only be realized in different experiments. We come back to this topic in Section 3, at the end of Section 6 and in an example in Section 7.

For a unital process  the Hilbert space $\qH$ is always the orthogonal sum of the ranges of the $V_\alpha$ for all $\alpha$ with a given length, so in principle the problem can be dealt with for an arbitrary vector state by decomposing the state vector with respect to such an orthogonal sum and then using Proposition \ref{prop:operational}. This may not always be the most practical path to follow for data given in a different way but it is of theoretical significance. 
It is of course interesting to prepare the system in other states where the outcomes are not deterministic but only statistical information can be obtained, see for example \cite{GS04} which uses discretization of continuous processes and thus can be directly connected with the approach here. 
The quantum filtering equations mentioned above provide recursive equations based on data from the Hamiltonian or the Lindblad generator of the quantum dynamical system.

We have established the operational meaning of the $\alpha$ in $V_\alpha$ in physical applications. In the following section we see further how an analysis of the recursive structure leads us naturally to the noncommutative Fornasini-Marchesini systems described in Section \ref{section:intro}

\section{Representations of Structure Maps}
\label{section:structure}

If for the multi-variable systems introduced in Section 1 %\ref{section:intro}
the spaces are Hilbert spaces and the linear maps are contractions between these Hilbert apaces then we want to think of them 
as appearing inside the weak Markov processes 
introduced in Section 2. 
%\ref{section:weak}
The following definitions give a precise meaning to that. Then we justify the definitions by discussing how the represented structure maps can help us to understand the properties of the process. 
Recall our discussion of multi-variable systems in Section 1 %\ref{section:intro} 
and the definition of the subspace $\qE$ in Proposition \ref{prop:wandering}.

\begin{Definition} \normalfont \label{def:structure}
A {\it representation of an input pair} $(A,B)$ (with column contractions $A: \qX \rightarrow \bigoplus^d_1 \qX$ and
$B: \qU \rightarrow \bigoplus^d_1 \qX$) is a process $(\qH, V, \fh)$ 
such that $\qX = \fh$ and $A = V^* |_{\fh}$, together with an isometry
$i_0: \qU \rightarrow \qE$ such that $B = V^*\, i_0$. We call $\qU_0 := i_0(\qU)$ the represented input space.

A {\it representation of an output pair} $(A,C)$ (with column contractions $A: \qX \rightarrow \bigoplus^d_1 \qX$ and 
$C: \qX \rightarrow \qY$) 
is a process $(\qH, V, \fh)$ 
such that $\qX = \fh$ and $A = V^* |_{\fh}$, together
with an isometry
$j_0: \qY \rightarrow \fh_1 = \fh \oplus \qE$ such that with the represented output space $\qY_0 := j_0(\qY)$ we have $C = j^*_0 P_{\qY_0} |_{\fh}$.

A {\it representation of the structure maps} $(A,B,C,D)$, 
where $D: \qU \rightarrow \qY$,
is given by representations of $(A,B)$ as an input pair and $(A,C)$ as an output pair as above with the same process $(\qH, V, \fh)$
such that $D = j^*_0 P_{\qY_0} |_{\qU_0} i_0$.
\end{Definition}

In the following we often suppress the isometries $i_0$ and $j_0$ and treat them as identifications whenever this simplifies the notation. 
Note that a represented input space $\,\qU_0$ is always wandering because by definition it is a subspace of the wandering subspace $\qE$, see Proposition \ref{prop:wandering}. 
A represented output space is in general not wandering; if it is we call it a wandering output space. 
An interesting example for a wandering output space is 
\[
\qE_* := \fh_1
\ominus V(\fh \otimes {\mathcal P})
= \qH \ominus V(\qH \otimes {\mathcal P}) = ker\, V^*\,.
\]
Indeed, this is the wandering subspace arising from the Wold decomposition of the row isometry $V$, see \cite{Po89a}.

Given any contractive block matrix of the form
\[
\left(
\begin{array}{cc}
A & B \\
C & D \\
\end{array}
\right) 
\colon
\left(
\begin{array}{c}
\fh \\
\qU \\
\end{array}
\right)
\rightarrow
\left(
\begin{array}{c}
\fh \otimes \qP \\
\qY \\
\end{array}
\right)
\]
we can use dilation theory to find a process so that $(A,B,C,D)$ is represented by it. In this case $\qU = \qE$ and $\qY = \qE_*$, compare \cite{Po89a,BV05}. This is also closely related to the realization theory of (noncommutative) Schur functions
\cite{BGM06}. Our definition \ref{def:structure} is rather general and does not always produce a contractive block matrix. 
In fact, given a process $(\qH, V, \fh)$, any subspace of
$\qE$ interpreted as a represented input space gives rise to a representation of an input pair
$(A,B)$ and any subspace of $\fh_1$ interpreted as a represented output space gives rise to a representation of an output pair
$(A,C)$ and together we have a representation of structure maps $(A,B,C,D)$. 
The additional flexibility can be useful in applications to processes and we show now that some important parts of the theory are still valid.

Let us write $\qU_\beta$ for $V_\beta \qU_0$ and $\qY_\alpha$
for $V_\alpha \qY_0$. Further we denote by $\qU_+$ the closed linear span of all $\qU_\alpha$ and by $\qY_+$ the closed linear span of all $\qY_\alpha$ (with $\alpha \in F^+_d$). With suitable identifications we can think of $\qU_+$ as an orthogonal direct sum of copies of the input space $\qU$. Similarly, if $\qY_0$ is wandering, then $\qY_+$ is an orthogonal direct sum of copies of the output space $\qY$. We sometimes refer to these copies as $\alpha$-translated input resp. output spaces. 

Representations of structure maps as above are always {\it causal} in the sense that $\qY_0 \perp V_\beta \qU_0 = \qU_\beta$ for all $\beta \in F_d^+$ with $|\beta| \ge 1$. More generally it follows further that $\qY_\alpha \perp \qU_\beta$  whenever $|\alpha| < |\beta|$ which we also refer to as causality. 

Note that from a representation with $A = (A_1, \ldots, A_d)^t$ we get a Kraus decomposition for the transition operator of the process, namely 
\[
Z(x) = A^* \, x\! \otimes\! \eins \, A = \sum^d_{k=1} A^*_k x A_k
\]
(limit in the strong operator topology if $d=\infty$). 

\begin{Proposition} \normalfont \label{prop:system}
Given a representation of structure maps $(A,B,C,D)$ by a process $(\qH, V, \fh)$ let $\tilde{\xi} = \xi \oplus \bigoplus V_\alpha \eta_\alpha$ be an element of $\qH$, with $\xi \in \fh$ and $\eta_\alpha \in \qU_0$.  Recall that $p = P_\fh$ denotes the orthogonal projection from $\qH$ to $\fh$.

If (for all words $\alpha$)
\begin{eqnarray*}
x(0) &:=& \xi \\
x(\alpha) &:=& p\, V^*_\alpha \tilde{\xi} 
\quad \big(\Rightarrow V_\alpha\, x(\alpha) 
= P_{V_\alpha \fh} \tilde{\xi} \;\big)\\
u(\alpha) &:=& \eta_\alpha 
\quad\quad\;\; \big(\Rightarrow V_\alpha\, u(\alpha) 
= P_{\qU_\alpha} \tilde{\xi} \;\big)
\end{eqnarray*}
then we have (for all words $\alpha$ and generators $k=1,\ldots,d$)
\[
x(\alpha k) = A_k \,x(\alpha) + B_k \,u(\alpha).
\]
If further
$y(\alpha) := P_{\qY_0} V^*_\alpha \tilde{\xi} \quad
\big(\Rightarrow V_\alpha\, y(\alpha) 
= P_{\qY_\alpha} \tilde{\xi} \;\big)$
 then we have
\[
y(\alpha) = C \,x(\alpha) + D \,u(\alpha).
\]
Hence we get a noncommutative Fornasini-Marchesini system
(compare Section 1). 
%\ref{section:intro}
\end{Proposition}

\begin{proof}
The first assertion follows from
\begin{eqnarray*}
x(\alpha k) &=& p\, V^*_{\alpha k} \tilde{\xi} 
= p\, V^*_k V^*_\alpha \tilde{\xi} 
= p\, V^*_k P_{\fh \oplus \qU_0} V^*_\alpha \tilde{\xi}
\\
&=& V^*_k p\, V^*_\alpha \tilde{\xi}
+ V^*_k \eta_\alpha
= A_k \,x(\alpha) + B_k \,u(\alpha).
\end{eqnarray*}
Using causality we find that
\begin{eqnarray*}
y(\alpha) &=& P_{\qY_0} V^*_\alpha \tilde{\xi} 
= P_{\qY_0} P_{\fh \oplus \qU_0} V^*_\alpha \tilde{\xi} \\
&=& P_{\qY_0} \big[ x(\alpha) + u(\alpha) \big]
= C x(\alpha) + D u(\alpha).
\end{eqnarray*}
\end{proof}

%Note that the proof uses causality but does not use the fact that $\qY_0$ is wandering. Hence the result is valid for generalized (non-wandering but causal) output spaces.

%With the maximal choice $\qU_0 = \qE$ we have for any vector $\tilde{\xi} \in \qH$ that
%\[
%P_{V_\alpha \qH}\, \tilde{\xi} 
%=
%P_{V_\alpha \fh}\, \tilde{\xi} \oplus
%\bigoplus_{\beta \in F^+_d} P_{V_\alpha \qU_\beta}\, \tilde{\xi}
%=
%V_\alpha\, x(\alpha) \oplus
%\bigoplus_{\beta \in F^+_d} V_{\alpha \beta}\, u(\alpha \beta)\,,
%\]
%which is relevant information with respect to Proposition \ref{prop:operational} above. 

Continuing the discussion about the operational meaning of 
$\alpha \in F_d^+$ in the study of quantum dynamical systems at the end of Section 2, how can we interpret the noncommutative Fornasini-Marchesini system established in Proposition 
\ref{prop:system} ? Let $x \in \qB(\fh)$ which we interpret
as an operator $x p \in \qB(\qH)$.
Suppose that at time $0$ we have an initial state given
by a unit vector $\tilde{\xi} \in \qH$ as described in Proposition \ref{prop:system}. Then at time $n$ we find
\[
\langle \tilde{\xi}, J^{(n)}(x p) \,\tilde{\xi} \rangle 
= \sum_{|\beta|=n} \langle V^*_\beta \tilde{\xi}, x p \,V^*_\beta \tilde{\xi} \rangle\,.
\]
Now suppose further that at times $1, \ldots, n$ we performed a measurement of the observables $Y_1, \ldots, Y_n$ described in Proposition \ref{prop:operational} and obtained the results 
$\alpha = (\alpha_1, \ldots, \alpha_n)$. Then except for the one summand with $\beta = \alpha$ all the other summands in the sum above are inconsistent with the observations. According to the rules of quantum mechanics all
the other summands have to be removed and the one remaining normalized. Because $x p = p x p$ it follows that
$x(\alpha) = p V^*_\alpha \tilde{\xi}$ is an unnormalized state vector in $\fh$ which describes the state of $\qB(\fh)$ at time $n$ conditioned by $Y_1=\alpha_1, \ldots, Y_n = \alpha_n$. The first equation $x(\alpha k) = A_k \,x(\alpha) + B_k \,u(\alpha)$ of the noncommutative Fornasini-Marchesini system gives a recursion for these $x(\alpha)$ and should be compared with quantum filtering equations for conditional states as described in \cite{BGM06}.

To get an interpretation for the second equation 
$y(\alpha k) = C \,x(\alpha) + D \,u(\alpha)$ of the noncommutative Fornasini-Marchesini system we may assume
that as before we are interested in the system described by $\fh$ but it is not directly accessible. However we have access to the represented output space and use the $y(\alpha)$ instead of the $x(\alpha)$ as a resource of indirect information about $\fh$ which is available to us. To prepare a more detailed analysis in the following we work out some additional tools. Some specific situations of this type appear in sections 6 and 7. 

Let us first work out
the observability map $\qO_{C,A}$ and the transfer function $\qT$, already introduced in Section 1,
%\ref{section:intro} 
for represented structure maps. To simplify formulas let us for the moment identify $\qY_\alpha$ and $\qY$, denoted
$\qY_\alpha \simeq_\alpha \qY$. Then
\[
P _{\qY_\alpha} |_\fh = V_\alpha P_{\qY_0} V^*_\alpha |_\fh
\simeq_\alpha C A^\alpha
\]
in other words, we obtain the $\alpha$-entry $C A^\alpha$ of the observability map $\qO_{C,A}$ by projecting the represention $\fh$ of the internal space $\qX$ to the $\alpha$-translated output space $\qY_\alpha$. 

Similarly with $\qY_\alpha \simeq_\alpha \qY$ and 
$\qU_0 \simeq \qU$ we obtain
\[
P _{\qY_\alpha} |_{\qU_0} = V_\alpha P_{\qY_0} V^*_\alpha |_{\qU_0}
\simeq_\alpha
\left\{
\begin{array}{cc}
D & \text{if}\; \alpha = 0 \\ 
C\, B_\alpha & \text{if}\; |\alpha| = 1 \\
C\, A_{\alpha_r} \ldots A_{\alpha_2}\,B_{\alpha_1} & \text{if}\;  \alpha = \alpha_1 \ldots \alpha_r,\, r = |\alpha| \ge 2
\end{array} 
\right.
\]
In other words, we obtain the $\alpha$-coefficient $\qT^\alpha$ in the formal power series expansion of the transfer function $\qT$ 
(already introduced in Section 1)
%\ref{section:intro}
by projecting the represention $\qU_0$ of the input space $\qU$ to 
the $\alpha$-translated output space $\qY_\alpha$. Causality implies that further,
with identifications 
$\qY_\alpha \simeq_\alpha \qY$ and 
$\qU_\beta \simeq_\beta \qU$,
\[
P _{\qY_\alpha} |_{\qU_\beta} \simeq_{\alpha,\beta}
\left\{
\begin{array}{cc}
\qT^\sigma & \text{if}\; \alpha = \beta \sigma \\ 
0 & \text{otherwise}
\end{array}
\right.
\]
This pattern in the operator-valued kernel 
$\big( P _{\qY_\alpha} |_{\qU_\beta} \big)_{\alpha, \beta \in F^+_d}$ describes what is called a multi-analytic kernel. 
If $\qY_0$ is wandering then this operator-valued matrix corresponds to the contraction $P_{\qY_+} |_{\qU_+}$, for the orthogonal decompositions with respect to the translated output resp. input spaces, and it intertwines the row shifts on $\bigoplus_\beta \qU_\beta$
and $\bigoplus_\alpha \qY_\alpha$ which are obtained by restricting the row isometry $V$.
This is the defining property of a multi-analytic operator and we have verified it for the contraction $P_{\qY_+} |_{\qU_+}$. Multi-analytic operators
have been introduced and studied by Popescu \cite{Po89b,Po95}, the situation above involving transfer functions for pairs of wandering subspaces is worked out in more detail in \cite{Go12}. 

For $d=1$ multiplication with the transfer function is multiplication by an ordinary analytic function. But in a similar way for all $d \ge 1$, for a noncommutative Fornasini-Marchesini system, we can introduce the noncommutative $Z-$transforms
of the input string $\big(u(\beta)\big)$ and of the output string $\big(y(\alpha)\big)$ (see Proposition \ref{prop:system}) as formal power series
\begin{eqnarray*}
\hat{u}(z) = \hat{u}(z_1, \ldots, z_d) = \sum_\beta u(\beta) z^\beta \\
\hat{y}(z) = \hat{y}(z_1, \ldots, z_d) = \sum_\alpha y(\alpha) z^\alpha
\end{eqnarray*}
and then, with the initial condition $x(0)=0$, they are related via multiplication by the transfer function $\qT$:
\[
\hat{y}(z) = \qT(z)\,\hat{u}(z)\,.
\]
To verify the formula recall the convention
$z^\alpha = z_{\alpha_r} \ldots z_{\alpha_1}$
for $\alpha = \alpha_1 \ldots \alpha_r$ which
gives the multiplication rule 
$z^\sigma z^\beta = z^{\beta \sigma}$.

This leads to the term {\it transfer function} in control theory.
Similar connections between system trajectories and the dynamics of the ambient system have been worked out in \cite{BV05}, see also \cite{BSV05} for a commutative polydisk setting and \cite{BCU08} for a one-variable continuous time setting.  

%While we think of the restricted $V_k$ as left multiplication, mapping $\qU_\alpha$ and $\qY_\alpha$ to $\qU_{k \alpha}$ and $\qY_{k \alpha}$,
%we should think of $z_k$ as right multiplication $\alpha \mapsto \alpha k$ to get an intuitive interpretion of the input-output relation. 
%The convention for multiplying $z$-variables becomes plausible at this point: the transfer function represents an operator which is an intertwiner. 

\section{Categories of Processes and $\gamma$-Extensions}
\label{section:cat}

In this section we introduce additional concepts with the intention to describe substructures of processes. 
We say that $(\qG, V^\qG, \fg)$ is a {\it subprocess} of the process
$(\qH, V, \fh)$ if $\fg$ is a closed subspace of $\fh$ which is co-invariant for $V$ and $V^\qG = V |_{\qG}$ where $\qG = \overline{span} \{V_\alpha \fg \colon \alpha \in F_d^+ \}$. Note that $\fg$ is also co-invariant for $V^\qG$ and $(\qG, V^\qG, \fg)$ is a process in its own right. 

Given a subprocess $(\qG, V^\qG, \fg)$ of a process $(\qH, V, \fh)$
we can form the {\it quotient process} 
\[
(\qH, V, \fh)/(\qG, V^\qG, \fg) := (\qK, V^\qK, \fk)
\]
where $\fk := \fh \ominus \fg,\; \qK := \overline{span} \{V_\alpha \fk \colon \alpha \in F_d^+ \},\; V^\qK 
:= V |_{\qK}$. Let us check that 
$(\qK, V^\qK, \fk)$ is a process. Indeed, because $\fg$ is co-invariant
for $V$ we see that $\qK$ is contained in $\qH \ominus \fg$. Hence
$\qK \ominus \fk$ is contained in $\qH \ominus \fh$ which is a $V$-invariant subspace orthogonal to $\fk$. Hence $\fk$ is co-invariant for $V^\qK$ which proves our claim. Note that in general $\fk$ is not co-invariant for $V$. 
Considering adjoints we see that $V^{\qG*} =
V^* |_\qG$ but only $V^{\qK*} =
(P_\qK \otimes \eins_{\mathcal P}) V^* |_\qK$. So we need to distinguish carefully between subprocesses and quotient processes.

It is convenient to reformulate these concepts within a {\it category of processes} which we define now. The objects of the category are the processes with a common multiplicity space ${\mathcal P}$. A {\it morphism} from $(\qR, V^\qR, \fr)$ to $(\qS, V^\qS, \fs)$ is a contraction $t: \fr \rightarrow \fs$ which intertwines the adjoints of the row isometries, i.e. 
\[
V^{\qS*}  t = (t \otimes \eins_{\mathcal P}) V^{\qR*} |_{\fr}\,.
\]
(or written differently: $A^\qS_k\, t = t\, A^\qR_k$ for $k=1,\ldots,d$ and $A^\qS_k = V^{\qS*}_k |_\fs,\; A^\qR_k = V^{\qR*}_k |_\fr$). 
Composition of morphisms is given by composition of operators,
the identity morphism is given by the identity operator.

We add some immediate observations justifying this definition. First note that a morphism $t$ is an isomorphism if and only if $t$ is unitary and we get a reasonable meaning for two processes to be isomorphic. In the following we often identify isomorphic processes.

Also note that if we have a morphism given by a contraction 
$t: \fr \rightarrow \fs$
then $t$ can always be extended to a contraction
$T: \qR \rightarrow \qS$ such that $T |_{\fr} = t$ and
$\| T \| = \| t \|$ and
\[
V^{\qS*}  T = (T \otimes \eins_{\mathcal P}) V^{\qR*}\,,
\]
this is nothing but the commutant lifting theorem in the version of Popescu 
\cite{Po92}.
We call $T$ an {\it extended morphism} associated to $t$. We can also think of a morphism as the class of all extended morphisms associated to $t$, this attaches a global interpretation to it. 

Finally it follows from 
$V^{\qS*}  t = (t \otimes \eins_{\mathcal P}) V^{\qR*} |_{\fr}$
that $t (\fr)$ is co-invariant for $V^\qS$, and hence from a morphism $t$ from $(\qR, V^\qR, \fr)$ to $(\qS, V^\qS, \fs)$, by defining
$\fg$ to be the closure of $t (\fr)$ and $V^\qG := V^\qS |_{\qG}$, we always obtain a subprocess $(\qG, V^\qG, \fg)$ of $(\qS, V^\qS, \fs)$.

In fact, we can reformulate subprocesses and quotient processes in terms of morphisms as follows. If $(\qG, V^\qG, \fg)$ is a subprocess of a process $(\qH, V, \fh)$ then $\eins_\fg$, interpreted as an embedding of $\fg$ into $\fh$, is an isometric morphism:
\[
V^{*}  \eins_\fg = (\eins_\fg \otimes \eins_\qP) V^{\qG*} |_\fg
\] 
Conversely, given an isometric morphism $t$ from a process $(\qG, V^\qG, \fg)$ to a process $(\qH, V, \fh)$ we can interpret $t$ as an embedding $\eins_\fg$ and the properties of a morphism ensure that $\fg$ is co-invariant for $V$. With an extended morphism $\eins_\qG$, interpreted as an embedding of $\qG$ into $\qH$, we can also arrange that $V^\qG = V |_\qG$ (by uniqueness of a minimal isometric dilation up to unitary equivalence). So we get a subprocess.

For the corresponding quotient process $(\qK, V^\qK, \fk)$ we can check that the orthogonal projection $P_\fk: \fh \rightarrow \fk$ is a coisometric morphism:
\[
V^{\qK*}  P_\fk = (P_\fk \otimes \eins_\qP) V^{*} |_\fh
\] 
Conversely, if we have a coisometric morphism $t$ from $(\qH, V, \fh)$ to a process $(\qK, V^\qK, \fk)$ then we can interpret $t$ as an
orthogonal projection $P_\fk: \fh \rightarrow \fk$, then use the properties of a morphism to check that $\fg := \fh \ominus \fk$ is
co-invariant for $V$, giving rise to a subprocess so that the corresponding quotient process is the process $(\qK, V^\qK, \fk)$
we started from.
Note that $P_\fg$ and $\eins_\fk$ are not morphisms, in general.

The image of $\eins_\fg$ is equal to the kernel of $P_\fk$ and we can proceed to give a concise description of the situation as a short exact sequence of processes:
\[
\xymatrix{
0 \ar[r] & (\qG, V^\qG, \fg) \ar[r]^{\eins_\fg}  &
(\qH, V, \fh) \ar[r]^{P_\fk}  &
(\qK, V^\qK, \fk) \ar[r]  & 0
}
\]
Here $0$ stands for the process on the $0$-dimensional space which takes the role of a zero object in our category. In the following we suppress the embeddings in the notation whenever this is convenient.

\begin{Lemma}\normalfont \label{lem:gamma}
If a short exact sequence of processes is given as follows
\[
\xymatrix{
0 \ar[r] & (\qG, V^\qG, \fg) \ar[r]^{\eins_\fg}  &
(\qH, V, \fh) \ar[r]^{P_\fk}  &
(\qK, V^\qK, \fk) \ar[r]  & 0
}
\]
then the relative position of the wandering subspaces
\[
\qE^\fg := \overline{span} (\fg, V^\qG(\fg \otimes {\mathcal P})) \ominus \fg\,,
\]
\[
\qE^\fk_* := \overline{span} (\fk, V^\qK(\fk \otimes {\mathcal P})) \ominus V^\qK(\fk \otimes {\mathcal P})
\]
can be described by
\[
P_{\qG} |_{\qE^\fk_*} = P_{\qE^\fg} |_{\qE^\fk_*}, \quad
P_{\qK} |_{\qE^\fg} =  P_{\qE^\fk_*} |_{\qE^\fg}  = \big(P_{\qE^\fg} |_{\qE^\fk_*}\big)^* .
\]
More general, for all $\alpha \in F^+_d$
\begin{eqnarray*}
P_{\qG} |_{V_\alpha \qE^\fk_*} &=& P_{V_\alpha \qE^\fg} |_{V_\alpha \qE^\fk_*}
=V_\alpha\, P_{\qE^\fg} |_{\qE^\fk_*}\, V^*_\alpha |_ {V_\alpha \qE^\fk_*} \\
P_{\qK} |_{V_\alpha \qE^\fg} &=& P_{V_\alpha \qE^\fk_*} |_{V_\alpha \qE^\fg}
= V_\alpha\, P_{\qE^\fk_*} |_{\qE^\fg}\, V^*_\alpha |_ {V_\alpha \qE^\fg}.
\end{eqnarray*}
\end{Lemma}

\begin{proof}
We have $\qE^\fk_* \subset \fh \oplus \qE$ where
$\qE := \overline{span} (\fh, V (\fh \otimes {\mathcal P})) \ominus \fh$. Because $V^* \,\fh \oplus \qE = \fh \otimes \qP$ we conclude that $V^* \qE^\fk_* \subset \fh \otimes \qP$. But by the definition of $\qE^\fk_*$ we also have
\[
(P_\qK \otimes \eins_{\mathcal P}) V^* \qE^\fk_* = V^{\qK*} \qE^\fk_* = \{0\}
\]
and it follows that $V^* \qE^\fk_* \subset \fg \otimes \qP$.
Hence if $\xi \in \qE^\fg,\, \eta \in \qE^\fk_*$ and $\alpha \not=0$
then $V^*_\alpha \eta \subset \fg$ and
\[
\langle V_\alpha \xi, \eta \rangle = \langle \xi, V^*_\alpha \eta \rangle = 0
\]
and it follows that $\qE^\fk_* \perp V_\alpha \qE^\fg$ for all $\alpha \not=0$. Obviously also $\qE^\fk_* \perp \fg$ and so
$
P_{\qE^\fg} |_{\qE^\fk_*} = P_{\qG} |_{\qE^\fk_*}
$
which is the first formula we intended to prove. 

It is clear that
$\big(P_{\qE^\fg} |_{\qE^\fk_*}\big)^* = P_{\qE^\fk_*} |_{\qE^\fg}$. To get the equality $P_{\qK} |_{\qE^\fg} =  P_{\qE^\fk_*} |_{\qE^\fg}$ we argue as follows.
Consider $\qE^\fk := \overline{span} (\fk, V (\fk \otimes {\mathcal P})) \ominus \fk$. 
Because $\qE^\fk \perp \fh$ it follows that $\qE^\fk \subset \qE$ and because $\qE$ is wandering for $V$ we conclude
that $V_\alpha \qE^\fk \perp 
\fh \oplus \qE$ if $\alpha \not=0$ and hence
\[
P_\qK \big(\fh \oplus \qE \big)
= \fk \oplus \qE^\fk\,.
\]
Now, because $\qE^\fg \subset  \fh \oplus \qE$, we can compute
\[
P_\qK |_{\qE^\fg} = P_\qK P_{\fh \oplus \qE} |_{\qE^\fg}
= P_{\fk \oplus \qE^\fk} |_{\qE^\fg}
= P_{\qE^\fk_*} |_{\qE^\fg}\,.
\]
For the last equality above note that 
$\qE^\fk_* = \big(\fk \oplus \qE^\fk \big)
\ominus V (\fk \otimes {\mathcal P})$ but
$
\qE^\fg \subset 
\overline{span} \{ g, V (\fg \otimes {\mathcal P}) \}
\perp V (\fk \otimes {\mathcal P})
$.
\\
\\
Because $\qE^\fg$ and $\qE^\fk_*$ are wandering it is now immediate that
\[
V_\alpha \qE^\fg \perp V_\beta \qE^\fk_* \quad \text{if} \; \alpha \not= \beta\,.
\]
To get the first line of the general formulas we only have to note additionally that 
always $V_\beta \qE^\fk_* \perp \fg$. To get the second line we have to prove that additionally
\[
V_\alpha \qE^\fg \;\perp\; \qK \ominus \bigoplus_{\beta \in F^+_d} V_\beta \qE^\fk_* =: \qK_{res} \;\;\text{(residual part)}.
\]
Indeed, if $\alpha=0$ this follows from $P_{\qK} |_{\qE^\fg} =  P_{\qE^\fk_*} |_{\qE^\fg}$ shown above. Then for other $\alpha$ reduce it to the previous case by using the fact that $V^\qK |_{\qK_{res}}$ is a row unitary and hence
$V^* \qK_{res} = V^{\qK*} \qK_{res} \subset \qK_{res} \otimes \qP$.
\end{proof}

These observations suggest the following construction which allows us 
to classify the processes which can be obtained by short exact sequences, given the subprocess and the quotient process.

\begin{Definition}\normalfont \label{def:extension}
Given processes $(\qG, V^\qG, \fg)$ and $(\qK, V^\qK, \fk)$ (with a common multiplicity space ${\mathcal P}$) and any contraction $\gamma: \qE^\fk_* \rightarrow \qE^\fg$ we define the $\gamma$-{\it extension} 
\[
(\qG, V^\qG, \fg) \oplus_\gamma (\qK, V^\qK, \fk) := (\qH, V, \fh)
\]
where
\begin{eqnarray*}
\fh &:=& \fg \oplus \fk, \\
\qH &:=& \fg \oplus \qK \oplus \bigoplus_{\alpha\in F_d^+}  ({\mathcal D}_{\gamma^*})_\alpha\,. \\
V &:=& 
\left\{ 
\begin{array}{cl}
\big(
\eins_\fg \oplus 
\left(
\begin{array}{c}
\gamma^* \\
D_{\gamma^*} \\
\end{array}
\right)
\big)
\, V^\qG
 & \text{on}\; \fg \\
\quad\quad\quad\quad \quad\quad\quad
V^\qK & \text{on}\; \qK \\ 
\text{canonical row shift}
 & \text{on}\; \bigoplus_{\alpha\in F_d^+}  ({\mathcal D}_{\gamma^*})_\alpha \\
\end{array}
\right.
\end{eqnarray*}
\end{Definition}

We add the following explanations for this definition. $D_{\gamma^*} := \sqrt{\eins - \gamma \gamma^*}$ is the defect operator for $\gamma^*$ and ${\mathcal D}_{\gamma^*}$, the closure of its range, is the defect space. Then $\bigoplus_{\alpha\in F_d^+}  ({\mathcal D}_{\gamma^*})_\alpha$ is the orthogonal sum of a family of copies of ${\mathcal D}_{\gamma^*}$ indexed by the free semigroup $F_d^+$, so what we mean by the canonical row shift on this space is
just moving elements between these copies. Further, to explain the action of $V$ on $\fg$, note that $V^\qG$ maps $\fg$ into $\fg \oplus \qE^\fg$ and now $\eins_{\fg}$ acts as identity on $\fg$, the contraction $\gamma^*$ maps 
$\qE^\fg$ into $\qE^\fk_* \subset \qK$ and $D_{\gamma^*}$ maps
$\qE^\fg$ into ${\mathcal D}_{\gamma^*}$ which we interpret as
$({\mathcal D}_{\gamma^*})_0 \subset  \bigoplus_{\alpha\in F_d^+}  ({\mathcal D}_{\gamma^*})_\alpha$. Written explicitly as an operator matrix with respect to the direct sum
$\qH = \fg \oplus \qK \oplus \bigoplus_{\alpha\in F_d^+}  ({\mathcal D}_{\gamma^*})_\alpha$ we have
\[
V = 
\left(
\begin{array}{ccc} 
A^{\qG*} & 0 & 0 \\
\gamma^* D_{A^{\qG*}} & V^\qK & 0 \\
D_{\gamma^*} D_{A^{\qG*}} & 0 & \text{row shift} \\
\end{array}
\right).
\]

Note that for $\gamma = 0$ the $\gamma$-extension is nothing but the direct sum of the two processes $(\qG, V^\qG, \fg)$ and $(\qK, V^\qK, \fk)$. If $(\qK, V^\qK, \fk)$ is unital and hence
$\qE^\fk_* = \{0\}$ then this is the only possibility. 

\begin{Theorem}\normalfont \label{thm:extension-p}
Given processes $(\qG, V^\qG, \fg)$ and $(\qK, V^\qK, \fk)$ (with a common multiplicity space ${\mathcal P}$)
and a contraction $\gamma: \qE^\fk_* \rightarrow \qE^\fg$. Then
the $\gamma$-extension
\[
(\qG, V^\qG, \fg) \oplus_\gamma (\qK, V^\qK, \fk) 
\]
is a process,
\begin{eqnarray*}
\gamma &=& P_{\qE^\fg} |_{\qE^\fk_*} = P_{\qG} |_{\qE^\fk_*} \\
\gamma^* &=& P_{\qE^\fk_*} |_{\qE^\fg} = P_\qK |_{\qE^\fg}
\end{eqnarray*}
and we have a short exact sequence
\[
0 \rightarrow (\qG, V^\qG, \fg) 
\stackrel{\eins_{\fg}}{\rightarrow} 
(\qG, V^\qG, \fg) \oplus_\gamma (\qK, V^\qK, \fk)
\stackrel{P_\fk}{\rightarrow} (\qK, V^\qK, \fk) \rightarrow 0\;.
\]
In other words, $(\qG, V^\qG, \fg)$
is a subprocess and $(\qK, V^\qK, \fk)$ is a quotient process of the $\gamma$-extension.
\\
Conversely, in a short exact sequence
\[
\xymatrix{
0 \ar[r] & (\qG, V^\qG, \fg) \ar[r]^{\eins'_\fg}  &
(\qH', V', \fh') \ar[r]^{P'_\fk}  &
(\qK, V^\qK, \fk) \ar[r]  & 0
}\,,
\]
where $\eins'_{\fg}$ resp. $P'_\fk$ are an isometric resp. a coisometric 
morphism,
the process $(\qH', V', \fh')$ is isomorphic to a  $\gamma$-extension for a contraction $\gamma: \qE^\fk_* \rightarrow \qE^\fg$ such that moreover the corresponding extensions are equivalent in the sense that the following diagram commutes:
\[
\xymatrix{
0 \ar[r] 
& 
(\qG, V^\qG, \fg) \ar[r]^{\eins'_\fg} & (\qH', V', \fh') \ar[r]^{P'_\fk} 
&
(\qK, V^\qK, \fk) \ar[r] & 0
\\
0 \ar[r] 
& 
(\qG, V^\qG, \fg) \ar[r]^(0.35){\eins_\fg} 
\ar[u]_{=} 
&
(\qG, V^\qG, \fg) \oplus_\gamma (\qK, V^\qK, \fk)
\ar[u]^{j}_{isom.}
\ar[r]^(0.65){P_\fk} 
&
(\qK, V^\qK, \fk) \ar[r] \ar[u]_{=} 
& 0
}
\]
Two extensions with $\gamma_1, \gamma_2: \qE^\fk_* \rightarrow \qE^\fg$ are equivalent if and only if
$\gamma_1$ and $\gamma_2$ are equal. 
\end{Theorem}

We can summarize the theorem by saying that there is a one-to-one correspondence between equivalence classes of extensions of
the process $(\qK, V^\qK, \fk)$ by the process $(\qG, V^\qG, \fg)$ and contractions from $\qE^\fk_*$ to $\qE^\fg$ and this correspondence is given by the construction of $\gamma$-extensions.

\begin{proof}
Working within the $\gamma$-extension as defined above we have $\qE^\fg$ identified via the isometry
$
\left(
\begin{array}{c}
\gamma^* \\
D_{\gamma^*} \\
\end{array}
\right)
$ 
with a subspace of $\qE^\fk_* \oplus ({\mathcal D}_{\gamma^*})_0$. 
Note that
$\qE^\fk_*$ is the wandering subspace arising from the Wold decomposition of $V^\qK$, see \cite{Po89a},
and $({\mathcal D}_{\gamma^*})_0$
is another wandering subspace with all translates orthogonal to the translates of $\qE^\fk_*$. Hence the embedded $\qE^\fg$ is wandering for $V$. With this it is now easy to check that $V$ is a row isometry, that $(\qH, V, \fh)$ is a process with subprocess $(\qG, V^\qG, \fg)$ and quotient process $(\qK, V^\qK, \fk)$
and that $\gamma$ and $\gamma^*$ satisfy the formulas stated.

Now suppose that the process $(\qH', V', \fh')$
is given by a short exact sequence, i.e. as an extension of
$(\qK, V^\qK, \fk)$ by $(\qG, V^\qG, \fg)$. 
We define 
$\gamma := P_{\qE^\fg} |_{\qE^\fk_*} \colon \qE^\fk_*\rightarrow \qE^\fg$ and then form the corresponding $\gamma$-extension
$(\qG, V^\qG, \fg) \oplus_\gamma (\qK, V^\qK, \fk) = (\qH, V, \fh)$.
We verify that this yields an equivalent extension by constructing a unitary
$J^*: \qH' \rightarrow \qH$ which intertwines the row isometries $V'$ and $V$ and
which maps $\eins'_\fg \xi$ to $\eins_\fg \xi$ if $\xi \in \fg$
and $(P'_\qK)^* \eta$ to $(P_\qK)^* \eta$
if $\eta \in \qK$. In fact, then the adjoint 
$J: \qH \rightarrow \qH'$ is an extended morphism and its restriction $j: \fh \rightarrow \fh'$ is the isomorphism we look for.

To see that $J^*$ exists it is enough to check that the remaining parts can be matched correctly. We invoke the following lemma which is a standard tool in operator theory.

\begin{Lemma}\normalfont \label{lem:defect}
Let $\qL'$ and $\qL$ be Hilbert spaces and
$\qL_0$ a closed subspace of $\qL$. If 
$w: \qL' \rightarrow \qL$ is an isometry such that
$\qL = \overline{span} \{\qL_0, w \qL' \}$ then $w$ is unitarily equivalent to
\[
\left(
\begin{array}{c}
\rho \\
D_{\rho} \\
\end{array}
\right)
\colon
\qL' \rightarrow \qL_0 \oplus {\mathcal D}_{\rho}
\quad\quad \text{where}\;\rho = P_{\qL_0} w \,.
\]
\end{Lemma}
The unitary from $\qL$ to $\qL_0 \oplus {\mathcal D}_{\rho}$ needed in the lemma is the identity on $\qL_0$ and it is
$w \xi \mapsto 
\left(
\begin{array}{c}
\rho \,\xi\\
D_{\rho}\, \xi\\
\end{array}
\right)
$
for $\xi \in \qL'$. We apply Lemma \ref{lem:defect} with 
$\qL' = \qE^\fg$, 
with $w$ being the isometric embedding of $\qE^\fg$ into $\qH'$ and with
$\qL =  \overline{span} \{\qK, w \, \qE^\fg \}, \; \qL_0 = \qK$. Lemma \ref{lem:gamma} shows that $\rho = \gamma^*$ in this case
and moreover that the embeddings of translates $V'_\alpha\, \qE^\fg$ 
follow exactly the pattern exposed by the $\gamma$-extension.
Hence we can put all the pieces together and get $J^*$.

Finally, if $(\qG, V^\qG, \fg) \oplus_{\gamma_1} (\qK, V^\qK, \fk)$ and $(\qG, V^\qG, \fg) \oplus_{\gamma_2} (\qK, V^\qK, \fk)$ are equivalent extensions then the unitary intertwiner $J^*$ constructed above maps for each element of $\qE^\fg$ its first embedding to the second. The same happens to elements of $\qE^\fk_*$. Hence if we suppress the embeddings we find that 
$P_{\qE^\fg} |_{\qE^\fk_*}$ is the same operator in both cases, i.e., $\gamma_1 = \gamma_2$.
\end{proof}

Remark: It is instructive to look at the situation described in the previous theorem from the point of view of dilation theory. Then we start with a row contraction on $\fh = \fg \oplus \fk$ of the form
\[
\left(
\begin{array}{cc}
X & 0 \\
Y &  Z \\
\end{array}
\right)\,
\]
which is called a lifting
(in our application this is $p_\fh V |_{\fh \otimes \qP}$).
It is well known that in such a situation $Y$ must have the form
$(D_{Z^*})^* \gamma^* D_X$ with a contraction 
$\gamma: \qD_{Z^*} \rightarrow \qD_X$,
see \cite{FF90}, Chapter IV, Lemma 2.1 for $d=1$ and
\cite{DG11}, Prop. 3.1 for the general case.
Hence a $\gamma$-extension can also be thought of as a description of the structure of the minimal isometric dilation of such a row contraction.
Liftings and their dilations are studied in \cite{DG07,DG11,DGH} and the results can be interpreted in the language of processes which we use in this paper.  
\\

Concerning the meaning of this theory of subprocesses and quotient processes within the interpretation as quantum mechanical processes it is clear that much work still needs to be done. We obtain some indications how such applications may look like when we analyze in Sections 6 and 7 how to get information from a subprocess about the full process if we have suitable observability properties. 

\section{Cascades of Systems}
\label{section:cascades}

One of the things one can do with linear systems is to stick them together in various ways. The most basic way to do that is to take the output of one system $I$ and to use it as the input of another system $II$. The combined system is then called a {\it simple cascade}, see for example \cite{GLR06} for the classical theory ($d=1$). It also works for the noncommutative Fornasini-Marchesini systems we have been considering here. Such cascade connections of Fornasini-Marchesini systems are also analyzed in Section 4 of \cite{BGM05}.
We need a slight generalization where the input of system $II$ is obtained from the output of system $I$ by applying a transformation $\Gamma$ to it. We call this a $\Gamma$-{\it cascade
of systems}. (In fact it is not really a generalization because we could absorb $\Gamma$ into the output map of system I or into the input map of system II or treat the middle part as a system on its own. But the terminology above is convenient when below we consider representations by processes.) 

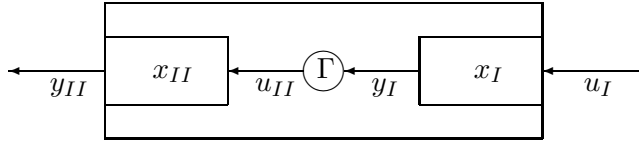
\setlength{\unitlength}{0.9cm}
\begin{picture}(15,2.5)

\put(4.4,0.2){\line(0,1){2}}
\put(4.4,0.2){\line(1,0){6.4}}
\put(10.8,0.2){\line(0,1){2}}
\put(4.4,2.2){\line(1,0){6.4}}

\put(4.4,0.7){\line(1,0){1.8}}
\put(4.4,1.7){\line(1,0){1.8}}
\put(6.2,0.7){\line(0,1){1}}
\put(5.1,1.1){$x_{II}$}

\put(9.0,0.7){\line(1,0){1.8}}
\put(9.0,1.7){\line(1,0){1.8}}
\put(9.0,0.7){\line(0,1){1}}
\put(9.8,1.1){$x_{I}$}

\put(4.4,1.2){\vector(-1,0){1.4}}
\put(9.0,1.2){\vector(-1,0){1.1}}
\put(7.3,1.2){\vector(-1,0){1.1}}
\put(12.2,1.2){\vector(-1,0){1.4}}

\put(7.6,1.2){\circle{0.6}}

\put(3.6,0.9){$y_{II}$}
\put(6.6,0.9){$u_{II}$}
\put(8.3,0.9){$y_{I}$}
\put(11.4,0.9){$u_{I}$}
\put(7.5,1.1){$\Gamma$}

\end{picture}

We assume here that the two noncommutative Fornasini-Marchesini systems both have the same multiplicity $d$. Then
the internal space of the combined system is defined to be the direct sum of the internal spaces of systems $I$ and $II$ and, 
with $u_{II}(\alpha) = \Gamma\, y_I(\alpha)$ (for all $\alpha \in F^+_d$), it follows, by eliminating variables, that the structure maps
$(A,B,C,D)$ of the combined system are obtained from the structure maps $(A^{I}, B^{I}, C^{I}, D^{I})$ of system $I$ and
$(A^{II}, B^{II}, C^{II}, D^{II})$ of system $II$ by
\[
A_j =
\left(
\begin{array}{cc}
A^{I}_j & \;0 \\
B^{II}_j \,\Gamma \,C^{I} & \;A^{II}_j
\end{array}
\right),
\quad
B_j =
\left(
\begin{array}{c}
B^{I}_j \\
B^{II}_j \,\Gamma\, D^{I} 
\end{array}
\right)
\quad\quad (j=1,\ldots,d),
\]
\[
C = 
\left(
\begin{array}{cc}
D^{II} \,\Gamma \,C^{I} & \;C^{II} \\
\end{array}
\right),
\quad
D = D^{II} \,\Gamma \,D^{I}\;.
\]
In this case we also speak of a $\Gamma$-cascade of structure maps.

The transfer function $\qT$ of such a $\Gamma$-cascade of systems (or of structure maps) factorizes. 
If $T$ is a power series with coefficients $T_\alpha \in \qB(V,W)$
and $\Gamma$ maps $W$ to $W'$ then we denote by $\Gamma T$
the power series with coefficients $\Gamma T_\alpha \in \qB(V,W')$.
With this convention it is not difficult to check that the transfer function $\qT$ of the combined system is obtained from the transfer functions $\qT^{II}$ and $\qT^{I}$ of systems $II$ and $I$ by
\[
\qT(z) =\qT^{II}(z)\, \Gamma \,\qT^{I}(z)
\]
(with $z^\alpha z^\beta = z^{\beta \alpha}$).
\\

Now we prove that if $(\qH, V, \fh)$ is a $\gamma$-extension $(\qG, V^\qG, \fg) \oplus_\gamma (\qK, V^\qK, \fk)$ of processes then we can think of its Fornasini-Marchesini system, from Proposition \ref{prop:system}, as a cascade of the systems associated to $(\qG, V^\qG, \fg)$ and $(\qK, V^\qK, \fk)$. This is not so obvious if we arrive at the notion of a $\gamma$-extension of processes from a dilation point of view and it gives an additional system theoretic motivation for the study of $\gamma$-extensions.

Given a $\gamma$-extension $(\qH, V, \fh) = (\qG, V^\qG, \fg) \oplus_\gamma (\qK, V^\qK, \fk)$. Suppose that we have a representation of an output pair  $(A^\qG, C^\qG)$ by $(\qG, V^\qG, \fg)$ with represented output space $\qY^\qG_0$. Because $\qY^\qG_0 \subset \fg \oplus \qE^\fg \subset \fh \oplus \qE$ 
%and $\qY^\qG_0$ is wandering for $V^\qG$ and hence for $V$ 
we can also think of $\qY_0 := \qY^\qG_0$ as an output space represented by $(\qH, V, \fh)$. Similarly suppose further that we have a representation of an input pair  $(A^\qK, B^\qK)$ by $(\qK, V^\qK, \fk)$ with represented input space $\qU^\qK_0$. Because
$\qU^\qK_0 \subset \qE^\fk \subset \qE$ we can also think of $\qU_0 := \qU^\qK_0$ as an input space represented by $(\qH, V, \fh)$.
We denote by $(A, B, C, D)$ the representation of structure maps in 
$(\qH, V, \fh)$ arising from $\qY_0 := \qY^\qG_0$ and $\qU_0 := \qU^\qK_0$.

To write $(A, B, C, D)$ as a $\Gamma$-cascade of structure maps we have to consider additionally an input space $\qU^\qG_0$ represented by $(\qG, V^\qG, \fg)$ and an output space $\qY^\qK_0$ represented by $(\qK, V^\qK, \fk)$. Now we have represented structure maps $(A^\qG, B^\qG, C^\qG, D^\qG)$ for $(\qG, V^\qG, \fg)$  and
$(A^\qK, B^\qK, C^\qK, D^\qK)$ for $(\qK, V^\qK, \fk)$, according to Definition \ref{def:structure}, and we have  
\[
\Gamma := P_{\qU^\qG_0}  |_{\qY^\qK_0}
= P_{\qU^\qG_0} \gamma P_{\qE^\fk_*} |_{\qY^\qK_0}
\]
where the latter equality follows from 
$\qU^\qG_0 \subset \qE^\fg$ together with
the geometry of the $\gamma$-extension (see Definition \ref{def:extension} and Lemma \ref{lem:gamma}). 

A subspace $\qL$ of $\qH$ is called a left support of $\gamma$ if
$P_\qL \gamma P_{\qE^\fk_*} = \gamma P_{\qE^\fk_*}$ and it is called a right support of $\gamma$ if
$\gamma P_{\qE^\fk_*} = \gamma P_{\qE^\fk_*} P_\qL$.
Roughly speaking, left and right supporting $\gamma$ 
means that the subspaces are chosen big enough to transport
the information contained in $\gamma$.
With these preparations we can now find a $\Gamma$-cascade of systems inside the $\gamma$-extension of processes. The main example is described in the corollary.

\begin{Theorem}\normalfont \label{thm:cascade-s}
Suppose that $\qU^\qG_0$ is a left support and 
$\qY^\qK_0$ is a right support of $\gamma$.
Then the structure maps 
$(A, B, C, D)$ are a $\Gamma$-cascade of
$(A^\qK, B^\qK, C^\qK, D^\qK)$ and $(A^\qG, B^\qG, C^\qG, D^\qG)$. 
Explicitly (for $j=1,\ldots,d$):
\[
A_j =
\left(
\begin{array}{cc}
A^\qK_j & \;0 \\
B^\qG_j \, \Gamma \,C^\qK & \;A^\qG_j\\
\end{array}
\right)
\quad
B_j =
\left(
\begin{array}{c}
B^\qK_j \\
B^\qG_j \, \Gamma \,D^\qK 
\end{array}
\right)
\]
\[
C = 
\left(
\begin{array}{cc}
D^{\qG} \,\Gamma \,C^{\qK} & \;C^{\qG} \\
\end{array}
\right),
\quad
D = D^{\qG} \,\Gamma D^{\qK}\;.
\]
\end{Theorem}

\begin{Corollary}\normalfont \label{cor:cascade-s}
With the choice $\qU^\qG_0 := \qE^\fg$ and $\qY^\qK_0 := \qE^\fk_*$
(or $\qY^\qK_0 := (ker\, \gamma)^\perp = \overline{\gamma^* \qE^\fg} \subset \qE^\fk_*$)
the assumptions of
Theorem \ref{thm:cascade-s} are satisfied and 
in this case we have $\Gamma = \gamma$. So we get a $\gamma$-cascade of systems (and of represented structure maps) and the transfer function of the $\gamma$-extension factorizes as follows:
\[
\qT^\qH(z) =\qT^\qG(z)\, \gamma \,\qT^\qK(z).
\]
\end{Corollary}

\begin{proof}
We verify the explicit formulas in Theorem \ref{thm:cascade-s} step by step. 
The arguments are based on the geometry of a $\gamma$-extension as given in Definition \ref{def:extension}, in particular: $V |_\fg = V^\qG |_\fg$ maps
$\fg$ into $\fg \oplus \qE^\fg$ and  $(\fg \oplus \qE^\fg) \ominus V^\qG (\fg \otimes \qP)$ is orthogonal to the range of $V^\qG$, further
$P_{\fg \oplus \qE^\fg} |_\qK =  P_{\qE^\fg} |_\qK
= \gamma P_{\qE^\fk_*} |_\qK$.
\begin{eqnarray*}
P_\fk A_j &=& P_\fk V^{*}_j |_\fh 
= 
\left(
\begin{array}{cc}
P_\fk V^{*}_j |_\fk & \;P_\fk V^{*}_j |_\fg \\
\end{array}
\right)
=
\left(
\begin{array}{cc}
A^\qK_j & \;0 \\
\end{array}
\right) 
\\
P_\fg A_j = P_\fg V^{*}_j |_\fh 
&= &
\left(
\begin{array}{cc}
P_\fg V^{*}_j |_\fk & \;P_\fg V^{*}_j |_\fg \\
\end{array}
\right)
= 
\left(
\begin{array}{cc}
V^{\qG *}_j P_{\fg \oplus \qE^\fg} |_\fk  & \;V^{\qG *}_j |_\fg \\
\end{array}
\right)
\\
&= &
\left(
\begin{array}{cc}
V^{\qG *}_j  P_ {\qE^\fg} \gamma P_ {\qE^\fk_*} |_\fk  & \;V^{\qG *}_j |_\fg \\
\end{array}
\right)
= 
\left(
\begin{array}{cc}
V^{\qG *}_j  P_{\qU^\qG_0} P_ {\qE^\fg} \gamma P_ {\qE^\fk_*} P_{\qY^\qK_0} |_\fk  & \;V^{\qG *}_j |_\fg \\
\end{array}
\right)
\\
&= &
\left(
\begin{array}{cc}
B^\qG_j \Gamma C^\qK  & \;A^\qG_j\\
\end{array}
\right) 
\end{eqnarray*}
\begin{eqnarray*}
P_\fk B_j &=& P_\fk V^{*}_j |_{\qU^\qK_0}
= P_\fk V^{\qK *}_j |_{\qU^\qK_0} = B^\qK_j \\
P_\fg B_j &=& P_\fg V^{*}_j |_{\qU^\qK_0} 
= P_\fg V^{\qG *}_j P_{\fg \oplus \qE^\fg} |_{\qU^\qK_0} 
= P_\fg V^{\qG *}_j P_{\qU^\qG_0} P_ {\qE^\fg} \gamma P_ {\qE^\fk_*} P_{\qY^\qK_0} |_{\qU^\qK_0}
= B^\qG_j \Gamma D^\qK 
\end{eqnarray*}
\begin{eqnarray*}
C &=& 
\left(
\begin{array}{cc}
P_{\qY^\qG_0} |_\fk &  \;P_{\qY^\qG_0} |_\fg \\
\end{array}
\right) 
=
\left(
\begin{array}{cc}
P_{\qY^\qG_0} P_{\qU^\qG_0} P_ {\qE^\fg} \gamma P_ {\qE^\fk_*} P_{\qY^\qK_0} |_\fk
&  \;P_{\qY^\qG_0} |_\fg 
\end{array}
\right) 
= 
\left(
\begin{array}{cc}
D^\qG \Gamma C^\qK & \;C^\qG \\
\end{array}
\right)
\\
D &=&  P_{\qY^\qH_0} |_{\qU^\qH_0}
= P_{\qY^\qG_0} |_{\qU^\qK_0}
= P_{\qY^\qG_0} P_{\qU^\qG_0} P_ {\qE^\fg} \gamma P_ {\qE^\fk_*} P_{\qY^\qK_0} |_{\qU^\qK_0}
= D^\qG \Gamma D^\qK \,.
\end{eqnarray*} 
This proves Theorem \ref{thm:cascade-s}. In the situation of Corollary \ref{cor:cascade-s} the left and right supporting property is clear and the remaining statements follow from the general discussion of cascades above.
\end{proof}

%Note that if we define, under the assumptions of Theorem \ref{thm:cascade-s}, the generalized structure maps
%\begin{eqnarray*}
%\hat{C}^\qK &:=& \Gamma C^\qK = P_{U^\qG_0} |_\fk, \\
%\hat{D}^\qK &:=& \Gamma D^\qK = P_{U^\qG_0} |_{U^\qK_0},
%\end{eqnarray*}
%then, written with these maps, the formulas in Theorem \ref{thm:cascade-s} look exactle like those for a simple cascade in 
%\cite{GLR06}. But these maps mix the $\qG$- and $\qK$-processes and hence are not represented structure maps in the sense of
%Definition \ref{def:structure}.

In the convenient situation of Corollary \ref{cor:cascade-s}
we still have the freedom to choose an output space $\qY_0 := \qY^\qG_0$ and an input space $\qU_0 := \qU^\qK_0$ according to our interests. An example is provided by
\[
\qY^\qG_0 := \qU^\qG_0 := \qE^\fg, \quad
\qU^\qK_0 := \qE^\fk\,.
\]
In this case $Y_0$ is wandering and we have
$C^\qG = 0$ and $D^\qG = \eins_{\qE^\fg}$ and 
\[
C = 
\left(
\begin{array}{cc}
\gamma \,C^{\qK} & 0 \\
\end{array}
\right),
\quad
D = \gamma D^{\qK}\;.
\]
Here the input-output system of the subprocess is trivial but the subprocess is used as a way to find an interesting output space for the
$\gamma$-extension. In a quantum physical process such a situation may occur if we confine our observations to the subprocess and try to learn from them about the extension. Examples of this type will be analyzed in more detail in the following sections.
\\

It is possible to iterate the construction shown in Theorem \ref{thm:cascade-s}.  Because  $V^* \qE^\fg_* = V^{\qG*} \qE^\fg_* = \{0\}$ we have $\qE^\fg_* \subset \qE_*$.
Hence if $(\qF, V^\qF, \ff)$ is another process (with the same multiplicity) and
$\gamma_2 \colon \qE^\fg_* \rightarrow \qE^\ff$ is a contraction then not only can we form
$(\qF, V^\qF, \ff) \oplus_{\gamma_2} (\qG, V^\qG, \fg)$ but also
$(\qF, V^\qF, \ff) \oplus_{\hat{\gamma}_2} \big[(\qG, V^\qG, \fg) \oplus_\gamma (\qK, V^\qK, \fk)\big]$ with $\hat{\gamma}_2:
\qE_* \rightarrow \qE^\ff$ given by
\[
\hat{\gamma}_2 :=
\left\{ 
\begin{array}{cl}
\gamma_2 & \text{on}\; \qE^\fg_* \\
0 & \text{on}\; \qE_* \ominus \qE^\fg_* \\
\end{array}
\right.
\]
Theorem \ref{thm:cascade-s} applies iteratively, for example with 
$
\qU^\qK_0 = \qE^\fk,\;
\qY^\qK_0 = \qE^\fk_*,\;
\qU^\qG_0 = \qE^\fg,\;
\qY^\qG_0 = \qE^\fg_*,\;
\qU^\qF_0 = \qE^\ff
$
etc.
Note that if the process $(\qG, V^\qG, \fg)$ is unital then 
$\qY^\qG_0 = \qE^\fg_* = \{0\}$ and the iteration is only possible as a direct sum of processes. On the other hand this observation suggests to study non-unital processes by considering extensions of this type. We leave this here as a future project. Relevant work is contained in section 6 of \cite{BC91} where decompositions of a given system into a cascade of two subsystems are constructed from invariant subspaces of $A= (A_1, \ldots, A_d)$. 

%This may be seen as another justification for calling such a structure a cascade. The constructive procedure from Section \ref{section:cat} for
%$\gamma$-cascades shows that there is a lot of flexibility in putting processes together in such a way, parametrized by contractions $\gamma_1, \gamma_2, \ldots$

\section{Observability}
\label{section:observable}

To make use of the system theory now available to us for the study of processes we discuss the control theory concept of observability in the multi-variable setting. See for example \cite{BBF11} for a recent treatment of the latter in a purely operator theoretic spirit.

\begin{Definition} \normalfont \label{def:observable}
Given an output pair $(A,C)$ for an internal space $\qX$ and an output space $\qY$, a subset $\qX' \subset \qX$
is called {\it observable} if 
$(C A^\alpha |_{\qX'})_{\alpha \in F^+_d}$, the observability map restricted to $\qX'$,
is injective
(as a map from $\qX'$ to the $\qY$-valued functions on $F^+_d$).

If $(A,C)$ is represented by the process 
$(\qH, V, \fh)$ then we also say in this case 
that $\qX'\, (\subset \fh)$ is observable in
$(\qH, V, \fh)$ by the output space $\qY_0$. 
\end{Definition}

The interpretation of observability is that every $\xi \in \qX'$ can be reconstructed from the outputs $C A^\alpha \xi$. 

\begin{Proposition} \normalfont \label{prop:observable}
The subset $\qX' \subset \fh$ is observable in $(\qH, V, \fh)$ by the output space $\qY_0$ if and only if
$P_{\qY_+} |_{\qX'}$ is injective.
\end{Proposition}

\begin{proof}
For $\xi \in \qX'$ we have
\[
C A^\alpha \xi = P_{\qY_0} V^*_\alpha \xi 
= V^*_\alpha P_{\qY_\alpha} \xi
\]
and we conclude that $\qX'$ is observable if and only if
$(P_{\qY_\alpha}  |_{\qX'})_{\alpha \in F^+_d}$ is injective.
Because the projection $P_{\qY_+}$ is the supremum of the projections $(P_{\qY_\alpha})_{\alpha \in F^+_d}$ we can
replace the family $(P_{\qY_\alpha}  |_{\qX'})_{\alpha \in F^+_d}$
by the single contraction $P_{\qY_+} |_{\qX'}$.
\end{proof}

We now concentrate on an important example already introduced in the previous section. 
If a process is a $\gamma$-extension and we use the maximal input space $\qE^\fg$ of the subprocess as a wandering output space for the extension, roughly speaking if we confine our observations to the subprocess, then the question of observability becomes a very natural issue of theoretical and practical importance.

\begin{Theorem} \normalfont \label{thm:observable}
Consider the $\gamma$-extension
\[
(\qG, V^\qG, \fg) \oplus_\gamma (\qK, V^\qK, \fk) = (\qH, V, \fh)
\]
and the output space $\qY_0 := \qE^\fg$.
The following assertions are equivalent:
\begin{itemize}
\item[(1a)]
$\fk$ is observable in $(\qH, V, \fh)$ by $\qY_0$.
\item[(1b)]
$P_{\qY_+} |_{\fk}$ is injective.
\item[(2a)]
$ \overline{span} \{(A^\alpha)^* \,\fg \colon \alpha \in F^+_d \} = \fh $
\item[(2b)]
$P_\fh |_\qG \colon \qG \rightarrow \fh$ has dense range.
\item[(3a)]
$\qG = \qH$
\item[(3b)]
There is an extended morphism associated to
the morphism $\eins_\fg: \fg \rightarrow \fh$
\\
which is a unitary from $\qH$ to $\qH$. 
\item[(4a)]
$V^\qK$ is a row shift and $\gamma \colon \qE^\fk_* \rightarrow \qE^\fg$ is injective.
\item[(4b)]
$V^\qK$ is a row shift and $\gamma \colon \qE^\fk_* \rightarrow \qE^\fg$ is isometric.
\end{itemize}
If the transition operator $Z$ of $(\qH, V, \fh)$ is unital then we also have the following equivalent condition:
\begin{itemize}
\item[(5)]
$\lim_{n \to \infty} Z^n (P_\fg) = \eins_\fh$ (in the strong operator topology)
\end{itemize}
\end{Theorem}

\begin{proof}
The equivalence of (1a) and (1b) is Proposition \ref{prop:observable}. 
Further
\begin{eqnarray*}
& & P_{\qY_+} |_{\fk}\; \text{injective} \\
&\Leftrightarrow &  P_\qG |_\fh \; \text{injective} 
\quad (\text{by adding the space} \, \fg) \\
&\Leftrightarrow &  P_\fh |_\qG \; \text{has dense range} 
\quad \text{(adjoint map)}
\end{eqnarray*}
which shows (1b) $\Leftrightarrow$ (2b). 
Because $G= \overline{span} \{V_\alpha \fg
\colon \alpha \in F^+_d \}$ assertion
(2b) is equivalent to
\[
\fh = 
\overline{span} \{ P_\fh V_\alpha \fg
\colon \alpha \in F^+_d \}
= \overline{span} \{ (A^\alpha)^* \fg
\colon \alpha \in F^+_d \}
\]
which is (2a). On the other hand
\[
P_\fh |_\qG \; \text{has dense range}
\Rightarrow  P_\qH |_\qG \; \text{has dense range}
\]
because $\qH$ is the closed linear span
of the $V_\alpha \fh$
with $\alpha \in F^+_d$ and
we have $\qG \supset V_\alpha \qG$, hence
\[
P_{V_\alpha \fh} \qG \supset P_{V_\alpha \fh} V_\alpha \qG
= V_\alpha P_\fh \qG
\]
which is dense in $V_\alpha \fh$. But $\qG$ is a closed subspace of $\qH$ and we conclude that (2b) implies (3a). The converse, (3a) implies (2b), is obvious.
 
It is easy to check that in a $\gamma$-extension the morphism $\eins_\fg$,
the embedding of $\fg$ into $\fh$, always has
$\eins_\qG$, the embedding of $\qG$ into $\qH$, as an associated extended morphism.
If we have (3a), i.e. $\qG = \qH$, then
$\eins_\qG$ is nothing but the identity operator $\eins_\qH$ on $\qH$, hence (3a) implies (3b). Conversely (3a) is implicit in the statement of (3b). 

From Definition \ref{def:extension} of the $\gamma$-extension we have
\[
\qE^\fg \subset \qE^\fk_* \oplus (D_{\gamma^*})_0
\]
with $\qE^\fk_* \subset \qK$ and
$(D_{\gamma^*})_0 \subset \qK^\perp$, 
from which we get $P_\qK \qY_+ \subset \bigoplus_\alpha V_\alpha \qE^\fk_*$. If we have (3a), i.e. $\qG = \qH$, then,
because $\qG = \fg \oplus \qY_+$ and
$\fg \perp \qK$, we find 
$P_\qK \qY_+ = \qK$
and hence $\qK = \bigoplus_\alpha V_\alpha \qE^\fk_*$ which means that $V^\qK$ is a row shift. From (3a) we have $\qK \subset \qG$
and together with $\qE^\fk_* \perp \fg$ and $\qE^\fk_* \perp
V_\alpha \qE^\fg$ for all $\alpha \not=0$ 
(by Lemma \ref{lem:gamma})
we conclude that $\qE^\fk_*$ is actually a subspace of $\qE^\fg$ which means that
$\gamma = P_{\qE^\fg} |_{\qE^\fk_*}$ is an isometric embedding.
Hence (3a) implies (4b). Obviously (4b) implies (4a) and from (4a), together with Lemma \ref{lem:gamma}, we get an injective map
\[
P_{\qY_+} |_\qK =
\bigoplus_\alpha V_\alpha \gamma V^*_\alpha \colon \quad
\qK = \bigoplus_\alpha V_\alpha \qE^\fk_* 
\rightarrow 
\bigoplus_\alpha V_\alpha \qE^\fg
= \qY_+
\]
and (1b) follows.

Finally if $Z$ is unital then the projections $\theta^n (P_\fg)$ increase with $n$ and converge to their supremum $P_\qG$
(see Section 2),
%\ref{section:weak}
hence  
$Z^n (P_\fg)= P_\fh \theta^n (P_\fg) |_\fh$ converges to $P_\fh P_\qG |_\fh$ (in the strong operator topology). Now (3a)
implies that 
$P_\fh P_\qG |_\fh = \eins_\fh$ and this implies (2b), hence all these assertions, including (5), are equivalent.  
\end{proof}

Remark: Of course (3b) does not mean that 
$(\qG, V^\qG, \fg)$ and $(\qH, V, \fh)$
are isomorphic as processes. In fact, if 
$\fg \not= \fh$ then the morphism $\eins_\fg$,
the embedding of $\fg$ into $\fh$, is not unitary. 

Let us simplify the terminology as follows.

\begin{Definition} \normalfont \label{def:ac}
If one (and hence all) of the assertions in Theorem
\ref{thm:observable} are satisfied
for the short exact sequence
\[
\xymatrix{
0 \ar[r] & (\qG, V^\qG, \fg) \ar[r]^{\eins_\fg}  &
(\qH, V, \fh) \ar[r]^{P_\fk}  &
(\qK, V^\qK, \fk) \ar[r]  & 0
}
\]
then we call this sequence observable or we say that the process $(\qH, V, \fh)$ is observable by the subprocess $(\qG, V^\qG, \fg)$.
\end{Definition}

%This terminology is borrowed from scattering theory. It will be explained in the final section of this paper how the setting above is related to already existing scattering theories for Markov chains. 

\begin{Corollary} \normalfont \label{cor:inner}
Let an observable sequence as in Definition \ref{def:ac} be given together with an input space $\,\qU_0$ and the wandering output space $\qY_0 = \qE^\fg$, both represented by $(\qH, V, \fh)$.
Then $\qO_{C,A} |_\fk = (C A^\alpha |_\fk)_{\alpha \in F^+_d}
= (C \, (A^\qK)^\alpha)_{\alpha \in F^+_d}$, the restriction of the observability map to $\fk$, is isometric and
the transfer function
$\qT$ is inner, i.e., the corresponding multi-analytic operator is isometric.
\end{Corollary}

\begin{proof}
Note that $\qY_0 = \qE^\fg \perp \fg$ indeed gives us
$C A^\alpha |_\fk = C (A^\qK)^\alpha$ for all $\alpha$.
We have $V_\alpha C A^\alpha |_\fk = P_{\qY_\alpha} |_{\fk}$. 
The multi-analytic operator corresponding to the transfer function $\qT$ is $P_{\qY_+} |_{\qU_+}$.

From observability it follows that $\fg \oplus \qY_+ = \qG = \qH$, see Theorem \ref{thm:observable} (3a).
Hence, because $\fk$ and $\qU_+$ are orthogonal to $\fg$, they are both subspaces of $\qY_+$ and it follows that $P_{\qY_+} |_{\fk}$ and $P_{\qY_+} |_{\qU_+}$ are isometric.
\end{proof} 

We describe a few alternative ways to interpret observability. First, if we have observability by a subprocess and we choose $\,\qU_0$ maximal, i.e.
$\,\qU_0 = \qE$, then the linear map given by the identity $\eins_\fg$ on $\fg$, by the observability map $\qO_{C,A} |_\fk$ on $\fk$ and by the
multi-analytic operator associated to $\qT$
on $\qU_+$ is equal to the identity $\eins_\qH$. This is a direct consequence of Corollary \ref{cor:inner}.
But it is the identity $\eins_\qH$ presented with a change of basis that describes the relative position of the weak
filtrations of the process $(\qH,V, \fh)$ and of the process $(\qG,V^\qG,\fg)$. See also Section 7 where this is interpreted as a kind of M{\o}ller operator in the sense of scattering theory. 

Second, from the point of view of dilation theory, in the case of observability by a subprocess we deal with so called subisometric dilations. We don't go into this here, see \cite{DG11} for details.  

Third, observability by a subprocess can also be given an explicit quantum physical interpretation by reconsidering the observables discussed in Proposition \ref{prop:operational}.
Indeed, in the case of observability the Hilbert space $\qH$ is the closed span of the spaces $V^{(n)} \,\fg \otimes \bigotimes^n_1 \qP$ (with $n \in \Nset_0$, the case $n=0$ to be interpreted as $\fg$). Let us assume that $(\qG,V^\qG,\fg)$
is a unital process. Then $V^{(n)} \,\fg \otimes \bigotimes^n_1 \qP$ is increasing with $n$ (see Section 2)
%\ref{section:weak}
and it follows that the algebras $V^{(n)} \,\qB(\fg \otimes \bigotimes^n_1 \! \qP) \,V^{(n)*}$ generate $\qB(\qH)$ (as a von Neumann algebra). Modifying the approach in Proposition \ref{prop:operational} we note that the observables
\begin{eqnarray*}
X_n &=& V^{(n)} \; X \! \otimes \bigotimes^{n}_1 \!
\eins_\qP \; V^{(n)*}\,, \\
Y_{m,n} &=& V^{(n)} \; \eins_\qH \otimes \bigotimes^{m-1}_1 \!
\eins_\qP \otimes Y \otimes \bigotimes^{n}_{m+1}\eins_\qP \; V^{(n)*}, \quad m=1,\ldots,n,
\end{eqnarray*}
if we consider all $X \in \qB(\fg) = P_\fg \qB(\fh) P_\fg$ and all $Y \in \qB(\qP)$, generate $V^{(n)}\, \qB(\fg \otimes \bigotimes^n_1 \! \qP)\, V^{(n)*}$ (as a von Neumann algebra) and we conclude that in the case of observability by the subprocess we can approximate arbitrary observables in $\qB(\qH)$ in the weak (or strong) operator topology by observables generated by these $X_n$ and $Y_{m,n}$
(which is the same as $X_n$ with $X \in \qB(\fg)$ and $Y_n$ with $Y \in \qB(\qP)$ for all $n \in \Nset_0$, see Proposition \ref{prop:operational}). 
In physics language,
we can answer all questions about observables on the part $\fk$
from measuring observables on the orthogonal part $\fg$ of the internal space (the $X_n$ with $X \in \qB(\fg)$) together with field observables (the $Y_n$). For example, in principle it is possible to determine the state from such observations. This is in fact a problem of quantum tomography. We cannot discuss this in detail at this point, see the survey article \cite{AGG05} for further information. In particular consider the important method of quantum homodyne tomography widely used in quantum optics, that is determining the state from measuring so called quadratures, different linear combinations of certain noncommuting observables in different experiments with the same unknown state. This can be realized in our scheme by varying the orthonormal basis chosen in $\qP$ for different experiments with the same unknown state and hence varying the $Y$-observable which is measured. We discuss a very specific example in Section 7.   

We have the following quantitative statement about these approximations which stresses the role of the observability operator $\qO_{C,A}$: If $\xi \in \fk$ then the norm distance squared of $\xi$ from the space $\fg_n = \fg \oplus \bigoplus_{|\alpha|<n} V_\alpha \qE^\fg$ is given by $\|\xi\|^2 - \sum_{|\alpha|<n} \|C A^\alpha \xi \|^2$. If we have observability then $\qO_{C,A} |_\fk$ is isometric and this distance tends to $0$ for $n \to \infty$. It is an interesting question how such formulas coming from the additive structure compare with the multiplicative structure involving tensor products and entanglement. We have to leave this investigation as a future project for now. 

\section{Subprocesses from Normal Invariant States}
\label{section:scattering}

The following way of finding subprocesses gives a connection to a topic which is of natural interest for quantum Markov processes and more general for quantum probability: invariant states. Recall that if $\phi$ is a normal state of a von Neumann algebra then in this von Neumann algebra there exists a smallest orthogonal projection $p$ such that $\phi(p)=1$, called the support projection $s(\phi)$. For all elements $x$ we have
\[
\phi(x) = \phi( s(\phi)\, x ) = \phi(x\, s(\phi) ) =\phi( s(\phi)\, x\, s(\phi) ) \,.
\]
Essentially the following is Lemma 6.1 from \cite{BJKW00}, for convenience we include a proof which uses our now familiar terminology and notation.

\begin{Proposition} \normalfont \label{prop:state}
Given a process $(\qH, V, \fh)$, suppose that $\phi$ is a normal state of $\qB(\fh)$ which is invariant for the transition operator $Z$, i.e.,
\[
\phi (Z(x)) = \phi(x) \quad \text{for all}\; x \in \qB(\fh)\,.
\]
Then with $P_\fg := s(\phi)$ the subspace $\fg$ is co-invariant for $V$.
\end{Proposition}

\begin{proof}
Let $\fk := \fh \ominus \fg$. Then 
\[
0 = \phi( P_\fk ) = \phi ( Z(P_\fk)) = \phi (P_\fg Z(P_\fk) P_\fg ),
\]
hence $P_\fg Z(P_\fk) P_\fg = 0$ (because $\phi$ restricted to
the subalgebra $s(\phi) \,\qB(\fh) \,s(\phi)$, obtained by compression with its support, is a faithful state). But
\[
P_\fg Z(P_\fk) P_\fg = P_\fg V (P_\fk \otimes \eins_\qP) V^* P_\fg
= X^*X
\]
with $X= (P_\fk \otimes \eins_\qP) V^* P_\fg$. Hence $X=0$ which shows that $\fg$ is co-invariant for $V$.
\end{proof}

Using Proposition \ref{prop:state} we can always find a subprocess from a normal invariant state and this subprocess is nontrivial, in the sense that $\fg \not= \fh$, if and only if the state is not faithful. Definition \ref{def:ac} for observability applies and we can use the criteria in Theorem \ref{thm:observable}. Instead of further analyzing the general case we concentrate for the rest of this section on an application to a class of noncommutative Markov processes which are not originally weak Markov processes but which nevertheless can be studied successfully by our methods.

The following construction represents the most basic way of introducing stationary Markov chains in an operator algebraic context. In this form it is taken from \cite{GKL06}, see more details there. A more leisurely introduction to the topic is \cite{Ku03}.
Let $\qA$ and $\qC$ be $C^*$-algebras and let
\[
j: \qA \rightarrow \qA \otimes \qC
\]
be a non-zero $*$-homomorphism. This can be iterated to yield
the $*$-homomorphisms
\[
j_n: \qA \rightarrow \qA \otimes \bigotimes^n_1 \qC\; ,
\]
where $j_1(a) := j(a) =: \sum_k  a_k \otimes c_k \in \qA \otimes \qC$ and then inductively
\[
j_n(a) := \sum_k j_{n-1}(a_k) \otimes c_k 
\in (\qA \otimes \bigotimes^{n-1}_1 \qC) \otimes \qC
= \qA \otimes \bigotimes^n_1 \qC\,.
\]
If $\phi$ resp. $\psi$ are states on $\qA$ respectively $\qC$ and we impose
the stationarity condition
\[
(\phi \otimes \psi) \circ j = \phi
\]
then we can think of the sequence $(j_n)$ as noncommutative random variables which form a noncommutative stationary Markov chain.

We can associate a weak process in the sense of this paper by applying the GNS-construction which from 
$(\qA, \phi)$ produces
$(\fh, \Omega_\phi)$ and from $(\qC, \psi)$ produces
$(\qP, \Omega_\psi)$, the GNS-Hilbert space and a cyclic vector representing the state (in each case). Then the stationarity condition translates into the fact that
\begin{eqnarray*}
v_1 \colon \quad \fh &\rightarrow& \fh \otimes \qP \\
a \Omega_\phi &\mapsto& j(a)\, \Omega_\phi \otimes \Omega_\psi \\
\end{eqnarray*}
(with $a \in \qA$) is an isometry. In Chapter 1 of \cite{Go04a} $v_1$ is called the {\it associated isometry} and it plays a central role there in the analysis of the Markov chain. We are now in a position to deepen this analysis
and to make the conceptual framework more elegant by putting it into the context provided in this paper. 
We start with the adjoint
\[
v^*_1: \fh \otimes \qP \rightarrow \fh
\]
which is a (row) coisometry. Hence its minimal isometric dilation
\[
V:  \qH \otimes \qP \rightarrow \qH
\]
is a row unitary and $(\qH, V, \fh)$ is a unital process. (In the notation used earlier in this paper and after a choice of basis in $\qP$  this is the minimal isometric dilation $V = (V_1, \ldots, V_d)$ of a coisometric row contraction $(A^*_1, \ldots, A^*_d)$.
It is a well known fact in dilation theory which can be checked directly that the minimal isometric dilation $V$ on $\qH$ of a row contraction
on $\fh$ is a row unitary if and only if the row contraction is coisometric.) We call $(\qH, V, \fh)$ the weak process {\it dual} to the original stationary Markov chain.
\\

Remark: If we think of $V^*$ as a kind of (Schr{\"o}dinger) dynamics of vector states then we see that for the dual weak process this is provided on $\fh$ by the associated isometry $v_1$. Note however that if we consider the original noncommutative random variables $j_n$ as steps of a (Heisenberg) dynamics of observables then there is a time reversal involved if instead we go for the noncommutative random variables $J^{(n)}$ (as in Section 2) %\ref{section:weak}
as steps of a (Heisenberg) dynamics of observables for the dual weak process. This is the reason why we call this weak process dual to the original chain. There is more to be said about this kind of duality, see \cite{Go04a}, 
Chapters 1 and 2, but here we just use this idea as an interesting way to produce weak processes.
\\

In fact there is more structure available from the stationary Markov chain. One can check that the vector state represented by $\Omega_\phi$ is invariant for the transition operator $Z'$ of the dual weak process and hence we can invoke Proposition \ref{prop:state} to see that its support, the one-dimensional subspace $\fg = \Cset \Omega_\phi$, is co-invariant for $V$. In fact it is easy to check this directly:
Consider the one-dimensional subspace $\fg := \Cset \Omega_\phi \subset \fh$. Because $V^* |_\fh = v_1$ and
$v_1 \Omega_\phi = \Omega_\phi \otimes \Omega_\psi$ we find
\[
V^* \fg = V^* \Cset \Omega_\phi = \Cset \Omega_\phi \otimes \Omega_\psi \subset \fg \otimes \qP
\]
which implies that $\fg$ is co-invariant for $V$ and hence we have a subprocess $(\qG, V^\qG, \fg)$. 
Combined with the analysis in Section 4
%\ref{section:cat} 
we obtain the surprising result that the (weak) process dual to a stationary Markov chain is automatically a $\gamma$-extension
\[
(\qH, V, \fh) = (\qG, V^\qG, \fg) \oplus_\gamma (\qK, V^\qK, \fk)
\]
and hence a whole short exact sequence
\[
0 \rightarrow (\qG, V^\qG, \fg) 
\stackrel{\eins_{\fg}}{\rightarrow} 
(\qG, V^\qG, \fg) \oplus_\gamma (\qK, V^\qK, \fk)
\stackrel{P_\fk}{\rightarrow} (\qK, V^\qK, \fk) \rightarrow 0\;.
\]
can be produced from the stationary Markov chain. We have
$\fg = \Cset \Omega_\phi$ and $\fk = \overline{\qA \Omega_\phi} \ominus \Cset \Omega_\phi$. This may potentially also be useful for the study of stationary states.
\\

{\bf Example:} Consider a (classical) Markov chain on a set with $3$ elements and a transition matrix
\[
T = 
\left(
\begin{array}{ccc}
\vspace{0.1cm}
\frac{1}{2} & \frac{1}{2} & 0 \\
\vspace{0.1cm}
\frac{1}{2} & 0           & \frac{1}{2} \\
\vspace{0.1cm}
0           & \frac{1}{2} & \frac{1}{2} \\
\end{array}
\right).
\]
The same example is also used in \cite{KM00} and \cite{GKL06} which allows further comparisons. To put it into the scheme introduced above we can implement it by a $*$-homomorphism
\begin{eqnarray*}
j \colon\quad \Cset^3 & \rightarrow & \Cset^3 \otimes \Cset^2 \\
           a & \mapsto &     \tilde{A}_1(a) \otimes \epsilon_1 + \tilde{A}_2(a) \otimes \epsilon_2\,.
\end{eqnarray*}
where $\Cset^3$ and $\Cset^2$ are considered as commutative algebras, with the canonical bases
$\{ \delta_1, \delta_2, \delta_3 \}$ and $\{ \epsilon_1, \epsilon_2 \}$, and
\[
\tilde{A}_1 = 				
\left(
\begin{array}{ccc}
0 & 1 & 0 \\
0 & 0 & 1 \\
0 & 0 & 1 \\
\end{array}
\right),
\quad\quad
\tilde{A}_2 = 				
\left(
\begin{array}{ccc}
1 & 0 & 0 \\
1 & 0 & 0 \\
0 & 1 & 0 \\
\end{array}
\right),
\]
together with a state $\psi$ on $\Cset^2$ induced by the probability measure $(\frac{1}{2},\frac{1}{2})$
on an underlying set of two elements. In fact, it is easy to check that
\[
T(a) = (\id \otimes \psi) \, j(a)\,.
\]
The probability measure $(\frac{1}{3},\frac{1}{3},\frac{1}{3})$ is invariant for $T$, so it induces a stationary state $\phi$ on $\Cset^3$. The cyclic vectors are $\Omega_\phi = \frac{1}{\sqrt{3}} (1.1.1)^t$and $\Omega_\psi = \frac{1}{\sqrt{2}} (1.1)^t$. We can identify the induced inner product on $\Cset^3$
and $\Cset^2$ with the canonical one, hence we put $\fh := \Cset^3$ and $\qP := \Cset^2$. The associated isometry is
\begin{eqnarray*}
v_1 \colon\quad \Cset^3 & \rightarrow & \Cset^3 \otimes \Cset^2 \\
           \xi\; & \mapsto &     A_1 \xi \otimes \epsilon_1 + A_2 \xi \otimes \epsilon_2
\end{eqnarray*}
with $A_1 = \frac{1}{\sqrt{2}}\, \tilde{A}_1$ and $A_2 = \frac{1}{\sqrt{2}}\, \tilde{A}_2$.
We can form the dual extended transition operator $Z'$ as follows:
\begin{eqnarray*}
Z' \colon\quad \qB(\Cset^3) &\rightarrow& \qB(\Cset^3) \\
    x           &\mapsto& \sum^2_{k=1} A^*_k x A_k
\end{eqnarray*}
which is unital completely positive and maps diagonal matrices to diagonal matrices in such a way that on the commutative subalgebra of diagonal matrices it reproduces the transition operator $T$. 
(As discussed above there is in general a time reversal involved here but in this specific example we get the same transition matrix.) Note also that the vector state given by $\Omega_\phi$ is invariant for $Z'$. 

As worked out in Section 2 we can define a (unital discrete weak Markov) process by a row unitary $V =(V_1, V_2)$, determined by $V_k |_\fh: \fh \rightarrow \fh \oplus \qE$, $\,k=1,2$, where $\qE$ is a $3$-dimensional Hilbert space for which we also fix an orthonormal basis, as follows:
\begin{eqnarray*}
V_1 |_\fh := 
\left(
\begin{array}{cc}
A^*_1 \\
B^*_1 \\
\end{array}
\right):
\fh \rightarrow \fh \oplus \qE \\
V_2 |_\fh := 
\left(
\begin{array}{cc}
A^*_2 \\
B^*_2 \\
\end{array}
\right):
\fh \rightarrow \fh \oplus \qE \\ 
\end{eqnarray*}
with $A_1, A_2$ as above and
\[
B_1 = \frac{1}{\sqrt{2}}				
\left(
\begin{array}{ccc}
1 & 0 & 0 \\
0 & 1 & 0 \\
0 & -1 & 0 \\
\end{array}
\right),
\quad\quad
B_2 = \frac{1}{\sqrt{2}}				
\left(
\begin{array}{ccc}
0 & 0 & 1 \\
0 & 0 & -1 \\
-1 & 0 & 0 \\
\end{array}
\right).
\]
It is easily checked that indeed $v_1 = V^* |_\fh$. So this is the weak Markov process dual to the original Markov chain. 

We can use $\qE$ as an input space.
Suppose that, as in Proposition \ref{prop:system}, we know the decomposition of the initial state vector at time $0$\,:
\[
\tilde{\xi} := \xi \oplus \bigoplus_{\alpha \in F^+_d} u(\alpha) \in \fh \oplus \bigoplus_{\alpha \in F^+_d} \qE = \qH\,,
\]
then, as shown in Proposition \ref{prop:system}, with the matrices specified above we can recursively compute the conditional state vectors $x(\alpha)$:
\[
x(0) = \xi, \quad x(\alpha k) = A_k \,x(\alpha) + B_k \,u(\alpha) \quad (k=1,2),
\]
conditioned on a measurement protocol $\alpha = (\alpha_1, \ldots, \alpha_n)$ which is obtained from measurements
$Y_1=\alpha_1, \ldots, Y_n = \alpha_n$ of observables 
$Y_m = V^{(m)} \; \eins_\qH \! \otimes \bigotimes^{m-1}_1 \!
\eins_\qP \otimes Y \; V^{(m)*},\; Y \epsilon_k = k\,\epsilon_k$, compare Propositions
\ref{prop:operational} and \ref{prop:system}. This is a version of quantum filtering. 

As discussed above we also have a subprocess based on $\fg = \Cset \Omega_\phi \subset \fh$ which gives us an interesting wandering subspace $\qE^{\fg}$ to use as an output space. With a short computation we find inside the $6$-dimensional space $\fh \oplus \qE$
\begin{eqnarray*}
\Omega_\phi &=& \frac{1}{\sqrt{3}} \big[ (1,1,1) \oplus (0,0,0) \big]^t, \\
V_1 \Omega_\phi &=& \frac{1}{\sqrt{6}} \big[ (0,1,2) \oplus (1,0,0) \big]^t, \\
V_2 \Omega_\phi &=& \frac{1}{\sqrt{6}} \big[ (2,1,0) \oplus (-1,0,0) \big]^t, \\
\end{eqnarray*}
so the $1$-dimensional space $\qE^{\fg}$ is spanned by the unit vector $\frac{1}{\sqrt{3}} \big[ (-1,0,1) \oplus (1,0,0) \big]^t$ and (using it as a basis for $\qE^{\fg}$)
\begin{eqnarray*}
C = P_{\qE^{\fg}} |_\fh &=& \frac{1}{\sqrt{3}} (-1,0,1)\,, \\
D = P_{\qE^{\fg}} |_\qE &=& \frac{1}{\sqrt{3}} (1,0,0)\,. 
\end{eqnarray*}
The dual extended transition operator $Z'$ is ergodic, i.e., its fixed point set is equal to $\Cset \eins_\fh$. This can be checked directly or the result can be taken from \cite{GKL06}. It follows that the process $(\qH, V, \fh)$ is observable by the subprocess $(\qG, V^\qG, \fg)$,
indeed for $\fg = \Cset \Omega_\phi$ the criterion 
$\lim_{n \to \infty} Z^n (P_\fg) = \eins_\fh$ for observability in
Theorem \ref{thm:observable}(5) is equivalent to ergodicity of $Z$, see
\cite{GKL06}, Section 3, or \cite{Go04a}, A.5.2.

The rest of our argument works not only in the example but whenever we have observability by a subprocess $(\qG, V^\qG, \fg)$ with $1$-dimensional $\fg = \Cset \Omega_\phi$ coming from a normal invariant state as in Proposition \ref{prop:state}.
Suppose now that actually we don't know the initial state vector $\tilde{\xi} \in \qH$. Because $\qG = \qH$ by Theorem \ref{thm:observable}(3a)
it is clear that $\tilde{\xi}$ can also be written in the form
\[
\tilde{\xi} = c \Omega_\phi \oplus \bigoplus_{\alpha \in F^+_d} y(\alpha) \in \fg \oplus \bigoplus_{\alpha \in F^+_d} \qE^{\fg} = \qG
\]
(with a complex number $c$), so it can be fully investigated within the $\qG$-process. As discussed at the end of Section 6, to determine $\tilde{\xi}$ from observations is a quantum tomography problem. It is simplified here because $\fg$ is $1$-dimensional, so we don't need observables of the form $X_n$ but we only need observables of the form $Y_n,\; n \in \Nset_0$, compare
Proposition \ref{prop:operational}. This is a remarkable achievement 
because it means in particular that we can determine the original state
$x(0)=\xi$ of the system described by $\qB(\fh)$ from observables which for example in the quantum optics settings mentioned in \cite{BHJ09} may be interpreted as describing the field surrounding the system. In fact, if in some way 
we have succeeded to determine the decomposition of $\tilde{\xi}$ in the $G$-process, i.e., the $y(\alpha)$ for all $\alpha \in F^+_d$,
then we can recover $\xi$ and the $u(\alpha),\, \alpha \in F^+_d$,
with the usual input-output formalism of control theory, that is by solving the noncommutative Fornasini-Marchesini system for $\xi$ and the $u(\alpha),\, \alpha \in F^+_d$. This is possible precisely because
the observability operator $\qO_{C,A}$ is injective, as ensured by 
observability, compare \cite{BBF11}.

We end with some remarks indicating a connection of these results to scattering theory. We can make this precise by looking at a scattering theory for noncommutative Markov chains first introduced in 
\cite{KM00}, with many further developments documented in 
\cite{Go04a,Go04b,GKL06,GHK}. We verify that observability by a subprocess based on
$\fg = \Cset \Omega_\phi$ for a dual weak process, as discussed above, is equivalent
to asymptotic completeness of the scattering theory for the noncommutative Markov chain given by the $*$-homomorphisms
$(j_n)_{n\in\Nset}$ we started from. See the precise statement in Proposition \ref{prop:complete} below. 

Let us start by investigating further what observability means in this case. We can define the associated isometry $v_n$ for the $n$-th noncommutative random variable $j_n$ which, because it arises from the same  iterative procedure, can be expressed by the iteration $V^{(n)}$ of $V$. With $a \in \qA$
and $\xi := a \Omega_\phi \in \fh$ we have
\[
j_n(a)\, \Omega_\phi \otimes \bigotimes^n_1 \Omega_\psi
= v_n \xi = V^{(n)*} \xi
\]
and from that
\begin{eqnarray*}
P_{\Omega_\phi \otimes \bigotimes^n_1\! \qP}\;
j_n(a) \,\Omega_\phi \otimes \bigotimes^n_1 \Omega_\psi
&=& P_{\fg \otimes \bigotimes^n_1\! \qP}\; v_n \xi \\
= V^{(n)*} V^{(n)} P_{\fg \otimes \bigotimes^n_1\! \qP} 
V^{(n)*} \xi
&=& V^{(n)*} P_{V^{(n)}\,\fg \otimes \bigotimes^n_1\! \qP}\; \xi
\end{eqnarray*}
which yields the norm equality
\[
\| P_{\fg \otimes \bigotimes^n_1\! \qP}\; v_n \xi \|
= \| P_{V^{(n)}\,\fg \otimes \bigotimes^n_1\! \qP}\; \xi \|\,.
\]
Because $P_{V^{(n)}\,\fg \otimes \bigotimes^n_1\! \qP}$ increases
to the projection $P_\qG$ (which is the limit for $n \to \infty$ in the strong operator topology) it follows that the property $\qG = \qH$, equivalent to observability by Theorem \ref{thm:observable}(3a),
is also equivalent to
\[
\| P_{\fg \otimes \bigotimes^n_1\! \qP}\; v_n \xi \|
\to \| \xi \| \quad \text{for}\; n \to \infty
\]
for all $\xi \in \fh$. If the GNS-representation is faithful then we can interpret the Hilbert space norm as a norm $\|\cdot \|_2$ on the $C^*$-algebra and we can write
\[
\| E_{\eins_\qA \otimes \bigotimes^n_1\! \qC}\;j_n(a) \|_2
\to \|a\|_2 \quad \text{for}\; n \to \infty
\]
for all $a \in \qA$, where $E_{\eins_\qA \otimes \bigotimes^n_1\! \qC}$ denotes the conditional expectation
obtained by evaluating the state $\phi$ on $\qA$.  

However the latter condition is well known to be equivalent to the asymptotic completeness of the stationary Markov chain (with a faithful stationary state) in the scattering theory context introduced by K\"ummerer and Maassen in \cite{KM00}. As mentioned earlier here we followed a variant developed in \cite{GKL06}. The equivalence of the convergence
in the norm $\|\cdot\|_2$ above with the property of asymptotic completeness as defined in the scattering theory context is stated in \cite{KM00}, 3.3 or \cite{Go04a}, 2.6.4 or \cite{GKL06}, 1.5. Let us summarize the result of our arguments as follows:

\begin{Proposition} \normalfont \label{prop:complete}
The short exact sequence produced by a stationary Markov chain (with a faithful stationary state) is observable in the sense of Definition \ref{def:ac} if and only if the stationary Markov chain is asymptotically complete in the scattering theory meaning of \cite{KM00} or \cite{GKL06}. A necessary and sufficient criterion is given by
\[
\lim_{n \to \infty} (Z')^n (P_\fg) = \eins_\fh \quad \text{(in the strong operator topology)}
\]
\end{Proposition}

In fact, the last statement is nothing but
criterion (5) for observability from Theorem \ref{thm:observable} applied to this special situation. This reproduces a criterion for asymptotic completeness in terms of the dual extended transition operator $Z'$, see \cite{Go04a}, 2.7.4 or \cite{GKL06}, Section 4 (in particular Theorem 4.3 there).
As mentioned earlier, because in this case $\fg = \Cset \Omega_\phi$ is one-dimensional the criterion is also equivalent to the ergodicity of $Z'$, i.e., the fixed point set of $Z'$ being equal to $\Cset \eins_\fh$, see \cite{GKL06}, Section 3, for more details.  

Let us finish with a sketchy discussion of related work towards scattering theory, with the purpose of providing some context and directing the reader to the relevant literature. 
The definition of asymptotic completeness for stationary Markov chains (with a faithful invariant state) in \cite{KM00,Go04a,GKL06} is given in terms of intertwiners (M{\o}ller operators) between a free and a perturbed dynamics.
This requires a two-sided process (i.e., with time variable in $\Zset$ for discrete time) as in \cite{KM00,Go04a}
or the construction of a two-sided extension as in \cite{GKL06}. In the
version of \cite{Go04a}, 2.6.4, we consider the $C^*$-algebra
\[
\qA \otimes \bigotimes_{0\not=n \in \Zset} \qC
\] 
on which we have an automorphism $\alpha$ given by the time evolution of the Markov chain and an automorphism $\sigma$ which is nothing but the right tensor shift on the $\qC$'s, acting identically on $\qA$. Then asymptotic completeness is the existence of the
M{\o}ller operator $\lim_{n \to \infty} \sigma^{-n} \alpha^n$ (pointwise weak${^*}$-limit) as an
isomorphism between the weak closure of $\qA \otimes \bigotimes_{0\not=n \in \Zset} \qC$ and the weak closure of
$(\eins \otimes) \bigotimes_{0\not=n \in \Zset} \qC$. It is necessary here to go to weak closures (with respect to a faithful invariant state) and to work in the category of von Neumann algebras. Note that wave operators, as used for example in \cite{He72}, are the same or inverses of M{\o}ller operators, depending on context.
Note also that to get a full scattering theory and to define a scattering operator as a composition of forward and backward wave operators we have to do the same construction also for the time reversed dynamics.

It is a natural question how the corresponding scattering theory on the level of weak processes looks like, i.e., in the setting we used in this paper. This is closely related to the approach in \cite{BV05} which is a study of row isometries in the spirit of Lax-Phillips scattering theory (which also inspired the approach in \cite{KM00} towards a scattering theory for stationary Markov chains). 
Further results already exist in the situation of a dual weak process for which, by Proposition \ref{prop:complete}, observability corresponds to the asymptotic completeness in the von Neumann algebra setting discussed above. In this case a version of the M{\o}ller operator acting between weak processes is worked out and discussed
in \cite{Go04a}, 2.5.6-2.5.7. Note that in this case infinite tensor products inherited from the underlying algebras can be used also for the Hilbert spaces. The paper \cite{Go09} builds a bridge between this work
and \cite{BV05}. 

But in fact we have seen a version of such a M{\o}ller operator on the level of weak processes also in this paper, namely the change of basis described in the end of Section 6 as the first interpretation offered for observability.
%\ref{section:observable} 
It relates the weak filtration of $(\qH,V,\fh)$ to the weak filtration of $(\qG,V^\qG,\fg)$ and is given by 
the identity on the one-dimensional $\fg = \Cset \Omega_\phi$ plus the observability map
from $\fk$ to $\qY_+$ plus the multi-analytic operator associated to the transfer function (on $\qU_+$).
If we take into account the identification of $\qG$ with the inductive
limit of the sequence $\big(\fg \otimes \bigotimes^n_1 \! \qP \big)$
mentioned in Section 2
%\ref{section:weak} 
and note that because $\fg = \Cset \Omega_\phi$ is one-dimensional we can identify it further with an infinite tensor product $\bigotimes^\infty_1 \! \qP$ then we arrive essentially at the version of the scattering theory worked out in \cite{Go04a}, Chapter 2.
 
But the setting of this paper is more general. For example we could also consider stationary Markov chains which are not originally constructed by tensor products of algebras and nevertheless associate weak processes via GNS-construction and apply Proposition \ref{prop:state} to find subprocesses (with $\fg = \Cset \Omega_\phi$). Or we can study higher-dimensional co-invariant subspaces $\fg$. The systematic use of weak processes adds conceptual clarity to such investigations.

Let us finally mention yet another point of view that can be adopted here which starts from the remark at the end of Section 4 that
$\gamma$-extensions can also be considered as dilations of contractive liftings. This motivated research on the corresponding transfer functions, i.e., the multi-analytic parts of the M{\o}ller operators from scattering theory, under the heading of characteristic functions for contractive liftings. The case with a one-dimensional co-invariant subspace for the subprocess has been investigated in \cite{DG07} and in fact the example of a characteristic function explicitly computed in Section 7 of \cite{DG07} comes exactly from the dual weak process of the Markov chain on a set with $3$ elements which we also used as an illustration earlier in this section (to compare note that our matrices 
$A_1. A_2$ correspond to 
$A^*_1, A^*_2$ in the notation of \cite{DG07}). The general case of characteristic functions for contractive liftings is defined and studied in \cite{DG11} and there is significant progress on this topic in the recent \cite{DGH}. 

Acknowledgements:
This work has been partly funded by the EPSRC-Research Grant 
EP/G039275/1. We thank the referees for detailed and constructive remarks leading to substantial improvements of older versions.

\end{document}